\documentclass[a4paper,11pt,reqno]{amsart}
\usepackage{amsmath,amsfonts,amsthm,amssymb,color}
\usepackage[T1]{fontenc}
\usepackage{pdfsync}
\usepackage{csquotes}
\usepackage{graphicx}
\usepackage{pstricks}
\usepackage{lmodern}

\usepackage{tikz}

\newcommand{\mathbbm}[1]{\text{\usefont{U}{bbm}{m}{n}#1}}
\newcommand{\be}{\beta}

\newcommand{\id}{\mbox{Id}}

\newcommand{\1}{{\bf 1}}

\newcommand{\lati}{\tilde{\la}}
\newcommand{\betati}{\tilde{\beta}}

\newcommand{\R}{\mathbb R}

\newcommand{\cf}{\mathcal F}

\newcommand{\cw}{\mathcal W}

\newcommand{\al}{\alpha}

\newcommand{\la}{\lambda}

\newtheorem{theorem}{Theorem}[section]

\newtheorem{definition}[theorem]{Definition}

\newtheorem{lemma}[theorem]{Lemma}

\newtheorem{proposition}[theorem]{Proposition}

\theoremstyle{remark}
\newtheorem{remark}[theorem]{Remark}

\pgfdeclareshape{crosscircle}
{
  \inheritsavedanchors[from=circle] 
  \inheritanchorborder[from=circle]
  \inheritanchor[from=circle]{north}
  \inheritanchor[from=circle]{north west}
  \inheritanchor[from=circle]{north east}
  \inheritanchor[from=circle]{center}
  \inheritanchor[from=circle]{west}
  \inheritanchor[from=circle]{east}
  \inheritanchor[from=circle]{mid}
  \inheritanchor[from=circle]{mid west}
  \inheritanchor[from=circle]{mid east}
  \inheritanchor[from=circle]{base}
  \inheritanchor[from=circle]{base west}
  \inheritanchor[from=circle]{base east}
  \inheritanchor[from=circle]{south}
  \inheritanchor[from=circle]{south west}
  \inheritanchor[from=circle]{south east}
  \inheritbackgroundpath[from=circle]
  \foregroundpath{
    \centerpoint%
    \pgf@xc=\pgf@x%
    \pgf@yc=\pgf@y%
    \pgfutil@tempdima=\radius%
    \pgfmathsetlength{\pgf@xb}{\pgfkeysvalueof{/pgf/outer xsep}}%
    \pgfmathsetlength{\pgf@yb}{\pgfkeysvalueof{/pgf/outer ysep}}%
    \ifdim\pgf@xb<\pgf@yb%
      \advance\pgfutil@tempdima by-\pgf@yb%
    \else%
      \advance\pgfutil@tempdima by-\pgf@xb%
    \fi%
    \pgfpathmoveto{\pgfpointadd{\pgfqpoint{\pgf@xc}{\pgf@yc}}{\pgfqpoint{-0.707107\pgfutil@tempdima}{-0.707107\pgfutil@tempdima}}}
    \pgfpathlineto{\pgfpointadd{\pgfqpoint{\pgf@xc}{\pgf@yc}}{\pgfqpoint{0.707107\pgfutil@tempdima}{0.707107\pgfutil@tempdima}}}
    \pgfpathmoveto{\pgfpointadd{\pgfqpoint{\pgf@xc}{\pgf@yc}}{\pgfqpoint{-0.707107\pgfutil@tempdima}{0.707107\pgfutil@tempdima}}}
    \pgfpathlineto{\pgfpointadd{\pgfqpoint{\pgf@xc}{\pgf@yc}}{\pgfqpoint{0.707107\pgfutil@tempdima}{-0.707107\pgfutil@tempdima}}}
  }
}
\makeatother

\colorlet{symbols}{blue!90!black}
\colorlet{testcolor}{green!60!black}

\def\symbol#1{\textcolor{symbols}{#1}}
\def\1{\mathbf{\symbol{1}}}

\usetikzlibrary{shapes.misc}
\usetikzlibrary{shapes.symbols}
\usetikzlibrary{snakes}
\usetikzlibrary{decorations}
\usetikzlibrary{decorations.markings}

\def\drawx{\draw[-,solid] (-3pt,-3pt) -- (3pt,3pt);\draw[-,solid] (-3pt,3pt) -- (3pt,-3pt);}
\tikzset{
	root/.style={circle,fill=testcolor,inner sep=0pt, minimum size=2mm},
	dot/.style={circle,fill=black,inner sep=0pt, minimum size=1mm},
	var/.style={circle,fill=black!10,draw=black,inner sep=0pt, minimum size=2mm},
	dotred/.style={circle,fill=black!50,inner sep=0pt, minimum size=2mm},
	generic/.style={semithick,shorten >=1pt,shorten <=1pt},
	dist/.style={ultra thick,draw=testcolor,shorten >=1pt,shorten <=1pt},
	testfcn/.style={ultra thick,testcolor,shorten >=1pt,shorten <=1pt,<-},
	testfcnx/.style={ultra thick,testcolor,shorten >=1pt,shorten <=1pt,<-,
		postaction={decorate,decoration={markings,mark=at position 0.6 with {\drawx}}}},
	kprime/.style={semithick,shorten >=1pt,shorten <=1pt,densely dashed,->},
	kprimex/.style={semithick,shorten >=1pt,shorten <=1pt,densely dashed,->,
		postaction={decorate,decoration={markings,mark=at position 0.4 with {\drawx}}}},
	kernel/.style={semithick,shorten >=1pt,shorten <=1pt,->},
	multx/.style={shorten >=1pt,shorten <=1pt,
		postaction={decorate,decoration={markings,mark=at position 0.5 with {\drawx}}}},
	kernelx/.style={semithick,shorten >=1pt,shorten <=1pt,->,
		postaction={decorate,decoration={markings,mark=at position 0.4 with {\drawx}}}},
	kernel1/.style={->,semithick,shorten >=1pt,shorten <=1pt,postaction={decorate,decoration={markings,mark=at position 0.45 with {\draw[-] (0,-0.1) -- (0,0.1);}}}},
	kernel2/.style={->,semithick,shorten >=1pt,shorten <=1pt,postaction={decorate,decoration={markings,mark=at position 0.45 with {\draw[-] (0.05,-0.1) -- (0.05,0.1);\draw[-] (-0.05,-0.1) -- (-0.05,0.1);}}}},
	kernelBig/.style={semithick,shorten >=1pt,shorten <=1pt,decorate, decoration={zigzag,amplitude=1.5pt,segment length = 3pt,pre length=2pt,post length=2pt}},
	rho/.style={dotted,semithick,shorten >=1pt,shorten <=1pt},
	renorm/.style={shape=circle,fill=white,inner sep=1pt},
	labl/.style={shape=rectangle,fill=white,inner sep=1pt},
	xi/.style={circle,fill=symbols!10,draw=symbols,inner sep=0pt,minimum size=1.2mm},
	xix/.style={crosscircle,fill=symbols!10,draw=symbols,inner sep=0pt,minimum size=1.2mm},
	xib/.style={circle,fill=symbols!10,draw=symbols,inner sep=0pt,minimum size=1.6mm},
	xibx/.style={crosscircle,fill=symbols!10,draw=symbols,inner sep=0pt,minimum size=1.6mm},
	not/.style={circle,fill=symbols,draw=symbols,inner sep=0pt,minimum size=0.5mm},
	>=stealth,
	}
\makeatletter
\def\DeclareSymbol#1#2#3{\expandafter\gdef\csname MH@symb@#1\endcsname{\tikz[baseline=#2,scale=0.15,draw=symbols]{#3}}\expandafter\gdef\csname MH@symb@#1s\endcsname{\scalebox{0.7}{\tikz[baseline=#2,scale=0.15,draw=symbols]{#3}}}}
\def\<#1>{\csname MH@symb@#1\endcsname}
\makeatother

\DeclareSymbol{circle}{0.5}{\draw (0,0.7) node[xi] {};}
\DeclareSymbol{line}{0.5}{\draw (0,0.2) node[not] {} -- (0,1.4) node[not] {};}

\DeclareSymbol{Psi}{0.5}{\draw (0,0) node[not] {} -- (0,1.5) node[xi] {};}
\DeclareSymbol{Psi2}{0.5}{\draw (-0.8,1) node[xi] {} -- (0,0) node[not] {} -- (0.8,1) node[xi] {};}
\DeclareSymbol{IPsi2}{0}{\draw (0,1) -- (0.8,2.2) node[xi] {};\draw (0,-0.25) node[not] {} -- (0,1) node[not] {} -- (-0.8,2.2) node[xi] {};}
\DeclareSymbol{PsiIPsi2}{0}{\draw (0,1) -- (0.8,2.2) node[xi] {};\draw (0,-0.25) node[not] {} -- (0,1) node[not] {} -- (-0.8,2.2) node[xi] {};\draw (0,-0.25) node[not] {} -- (0.8,1) node[xi] {};}

\newcommand{\luxor}{\mathbf{\Psi}}

\newcommand{\cherry}{\mathbf{\Psi}^{\mathbf{2}}}

\date{\today}

\title{Multilinear smoothing and local well-posedness of a stochastic quadratic nonlinear Schr\"{o}dinger equation}

\numberwithin{equation}{section}
\begin{document}

\begin{center}
{\large\textbf{
Multilinear smoothing and local well-posedness of a stochastic quadratic nonlinear Schr\"{o}dinger equation}}\\~\\ Nicolas Schaeffer\footnote{Institut \'Elie Cartan, Universit\' e de Lorraine, BP 70239, 54506 Vandoeuvre-l\`es-Nancy, France. }
\end{center}

\smallskip
{\small \noindent {\bf Abstract.}}
In this article, we study a $d$-dimensional stochastic  quadratic nonlinear Schr\"{o}dinger equation (SNLS), driven by a fractional derivative (of order $-\alpha<0$) of a space-time white noise:
\begin{equation*}
\left\{
\begin{array}{l}
i\partial_t u-\Delta u= \rho^2 |u|^2 + \langle \nabla \rangle^{-\alpha}\dot{W} \, , \quad  t\in [0,T] \, , \, x\in \R^d \, ,\\
u_0=\phi\, ,
\end{array}
\right.
\end{equation*}
where $\rho :\mathbb{R}^d \rightarrow \mathbb{R}$ is a smooth compactly-supported function. When $\alpha < \frac{d}{2}$, the stochastic convolution is a function of time with values in a negative-order Sobolev space and the model has to be interpreted in the Wick sense by means of a time-dependent renormalization. When $1\leq d \leq 3$, combining both the classical Strichartz estimates and a deterministic local smoothing, we establish the local well-posedness of (SNLS) for a small range of $\alpha$, in the spirit of \cite{Schaeffer1}. Then, we revisit our arguments and establish multilinear smoothing on the second order stochastic term. This allows us to improve our local well-posedness result for some $\alpha$. We point out that this is the first result concerning a Schr\"{o}dinger equation on $\mathbb{R}^d$ driven by such an irregular noise and whose local well-posedness results from both a stochastic multilinear smoothing and a deterministic local one combined with Strichartz inequalities.

\tableofcontents
\newpage
\section{Introduction and main results}

In this article, we study a $d$-dimensional stochastic  quadratic nonlinear Schr\"{o}dinger equation, driven by a fractional derivative (of order $-\alpha<0$) of a space-time white noise:
\begin{equation}\label{dim-d-sch}
\left\{
\begin{array}{l}
i\partial_t u-\Delta u= \rho^2 |u|^2 + \langle \nabla \rangle^{-\alpha}\dot{W} \, , \quad  t\in [0,T] \, , \, x\in \R^d \, ,\\
u_0=\phi\, ,
\end{array}
\right.
\end{equation}
where, for any $\alpha >0$, $\langle \nabla \rangle^{-\alpha}$ is the Bessel potential of order $\alpha$, $\dot{W}$ denotes a space-time white noise on $\mathbb{R}^d$ and $\rho :\mathbb{R}^d \rightarrow \mathbb{R}$ is a smooth compactly-supported function.\\
Schr\"{o}dinger equations with polynomial nonlinearities and forced by white noises have aroused the curiosity of researchers for many years by giving birth to numerous challenging questions and because of their applications in physics. A major objective in this field would be for instance to obtain the local well-posedness of the following equation:
\begin{equation*}
\left\{
\begin{array}{l}
i\partial_t u-\Delta u= |u|^2 u + \dot{W} \, , \quad  t\in [0,T] \, , \, x\in D \, ,\\
u_0=\phi\, ,
\end{array}
\right.
\end{equation*}
where $D=\mathbb{R}^d$ or $D=\mathbb{T}^d$. Indeed, the cubic nonlinear Schr\"{o}dinger equation describes for instance the propagation of light waves in optical fibers or can be used to study a simplified model of Bose-Einstein condensates \cite{Erdos} or even freak waves in the ocean \cite{Henderson}. The space-time white noise $\dot{W}$ represents a stochastic perturbation whose values at each space-time points are independent and thus allows a generalization of the classical model.\\
Some progress have been made in this direction. Forlano, Ho and Wang \cite{Forlano} have managed to prove the local well-posedness of a stochastic cubic nonlinear Schr\"{o}dinger equation with almost space-time white noise on the one-dimensional torus, namely:
\begin{equation*}
\left\{
\begin{array}{l}
i\partial_t u-\partial^2_x u= |u|^2 u + \langle \nabla \rangle^{-\alpha}\dot{W} \, , \quad  t\in [0,T] \, , \, x\in \mathbb{T} \, ,\\
u_0=\phi\, ,
\end{array}
\right.
\end{equation*}
with $\alpha >0$. In the following, let us denote by $\<Psi>$ the associated stochastic convolution, defined by, for all $t \in [0,T]$, $\displaystyle \<Psi>(t,.)=-i\int_{0}^t \frac{e^{-i(t-t')\Delta}}{\langle \nabla \rangle^{\alpha}}\dot{W}(t',.)dt'$. Their key arguments have been to measure the regularity of $\<Psi>$ in terms of Fourier-Lebesgue spaces and, after a suitable renormalization of the equation, to establish a fixed-point argument in Bourgain spaces thanks to new trilinear estimates. We underline the fact that, in their paper, the stochastic convolution has for spatial regularity $s<\alpha-\frac{1}{p}$ so that, given $\alpha>0$, they can choose sufficiently large $p$ in such a way that $\<Psi>$ has for spatial regularity $s>0$ and is consequently a function of time with values in a space of functions. Then, the previous work has been extended to some nonlinearity of order $p>1$, that is of the form $|u|^{p-1}u$ (see \cite{oh-okamoto-last,oh-popovnicu-wang-last}), on $\mathbb{R}^d$. Another model is the one proposed by Deya, Schaeffer and Thomann \cite{Schaeffer1} and whose dynamic is described through the equation below:
\begin{equation}\label{fbm-sch}
\left\{
\begin{array}{l}
i\partial_t u-\Delta u= \rho^2 |u|^2 + \dot B \, , \quad  t\in [0,T] \, , \, x\in \mathbb{R}^d \, ,\\
u_0=\phi\, ,
\end{array}
\right.
\end{equation}
for $1\leq d \leq 3$ and with $\rho:\mathbb{R}^d \rightarrow \mathbb{R}$ a smooth compactly-supported function. In their article, the authors have considered a quadratic nonlinearity and a very rough noise, namely the derivative of a fractional (in time and in space) Brownian field. In what they call the rough case, that is when the stochastic convolution $\<Psi>$ is only a function of time taking values in a space of distributions, there is a difficulty in defining $|\<Psi>|^2$. Resorting to a Wick renormalization combined with a local smoothing effect of the Schr\"{o}dinger operator, they have established the local well-posedness of (\ref{fbm-sch}) in Sobolev spaces of negative order, namely $H^{-2\alpha}(\mathbb{R}^d)$, for small $\alpha >0$. We will make a comparison between their model and ours in Subsection \ref{pointfixe1}. In fact, the litterature concerning Schr\"{o}dinger equations with nonlinearities and additive white noise is rather poor. The main reasons are that the stochastic convolution has  very low regularity and that the Schr\"{o}dinger operator, contrary to the heat or waves operator, does not offer any form of global regularization. To end with, let us mention the work of Deng, Nahmod and Yue \cite{Deng}. In order to study the propagation of randomness under nonlinear dispersive equations, they have developed the theory of random tensors. Thanks to these new tools, they have been able to prove local well-posedness for semilinear Schr\"{o}dinger equations in spaces that are subcritical in the probabilistic scaling.
Before coming back to our model, let us present two other works that are not dealing with Schr\"{o}dinger equations but with waves equations (another kind of dispersive equations) but that share an important common point with our model. The first one is the article of Gubinelli, Koch and Oh \cite{Gubi} in which they have used ideas from paracontrolled calculus to establish the local well-posedness of the following equation:
\begin{equation*}
\partial^2_t u+(1-\Delta) u= -|u|^2 +\dot{W} \, , \quad  t\in [0,T] \, , \, x\in \mathbb{T}^3 \, .
\end{equation*}
It is well-known that the waves operator provides a gain of one derivative (measured for instance in term of Sobolev norms). What interests us in this work is that the authors have proved a $\frac{1}{2}$-extra smoothing for the convolution between the waves operator and the second order stochastic term: $\displaystyle\<IPsi2>=(\partial^2_t + (1-\Delta))^{-1}\<Psi2>$ where $\<Psi2>$ denotes a renormalized version of $\<Psi>^2$. That was the first time that such a smoothing resulting from stochastic tools combined with multilinear dispersive analysis was obtained. In a similar manner, Oh and Okamoto \cite{Oh} have obtained a $\frac{1}{4}$-extra smoothing for the convolution between the waves operator and the second order stochastic term when studying the model below:
\begin{equation*}
\partial^2_t u+(1-\Delta) u= -|u|^2 + \langle \nabla \rangle^{\alpha}\dot{W} \, , \quad  t\in [0,T] \, , \, x\in \mathbb{T}^2, \,
\end{equation*}
with $\alpha>0$.
The main novelty of our work lies in the fact that, as in the wave setting, we will be able to prove an extra smoothing for the convolution between the Schr\"{o}dinger operator and the second order stochastic term (see Proposition \ref{cerise-ssj} for more details) thanks to which we will be able to establish the local well-posedness of our model for some range of $\alpha$. This is the first time that such a smoothing resulting both from stochastic and dispersive analysis is proved. 

\subsection{Interpretation of our model}

In this article, our aim is to establish the local well-posedness of the following equation:
\begin{equation}\label{model}
\left\{
\begin{array}{l}
i\partial_t u-\Delta u= \rho^2 |u|^2 + \langle \nabla \rangle^{-\alpha}\dot{W} \, , \quad  t\in [0,T] \, , \, x\in \R^d \, ,\\
u_0=\phi\, ,
\end{array}
\right.
\end{equation}
for some $\alpha>0$ and where $\rho :\mathbb{R}^d \rightarrow \mathbb{R}$ is a fixed smooth compactly-supported function. Here, $\dot{W}$ is a space-time white noise, that is a distribution-valued random variable such that for every test function $f$, $\langle \dot{W}, f \rangle$ is a centered Gaussian random variable with variance
$$\mathbb{E}[|\langle \dot{W}, f \rangle|^2]=\|f\|^2_{L^2([0,T]\times \mathbb{R}^d)}.$$
As usual, in order to isolate the expected worse term, we intend to resort to the Da Prato and Debbusche trick. We thus consider the stochastic convolution denoted by $\<Psi>$, the solution of the linear equation
\begin{equation*}
\left\{
\begin{array}{l}
i\partial_t \<Psi>-\Delta \<Psi>= \langle \nabla \rangle^{-\alpha}\dot{W} \, , \quad  t\in [0,T] \, , \, x\in \R^d \, ,\\
\<Psi>(0,.)=0\, .
\end{array}
\right.
\end{equation*}
Rewriting equation (\ref{model}) under the mild form, we see that $u$ is solution of
$$u_t=e^{-it\Delta}\phi-i\int_0^t e^{-i(t-\tau)\Delta}( \rho^2 |u_{\tau}|^2)\, d\tau+ \<Psi>\, , \quad t\in [0,T] \, .$$
Consequently, $v:=u-\<Psi>$ has to verify the equation below
\begin{multline}\label{equation sur v}
v_t=e^{-it\Delta}\phi-i\int_0^t e^{-i(t-\tau)\Delta}( \rho^2 |v_{\tau}|^2)\, d\tau-i\int_0^t e^{-i(t-\tau)\Delta}( (\rho \overline{v}_{\tau})\cdot (\rho \<Psi>_{\tau}))\, d\tau\\
-i\int_0^t e^{-i(t-\tau)\Delta}((\rho v_{\tau})\cdot(\overline{\rho \<Psi>_{\tau}}))\, d\tau-i\int_0^t e^{-i(t-\tau)\Delta}(|\rho \<Psi>_{\tau}|^2)\, d\tau\, , \quad t\in [0,T] \, .
\end{multline}
We see that there are three major obstacles in the treatment of the latter equation, namely the proper definitions of the products $\rho v \cdot \overline{\rho \<Psi>}$, $\rho \overline{v} \cdot \rho \<Psi>$ and $|\<Psi>|^2$ and how to deal with the quadratic term $\rho^2 |v|^2$. Let us first focus our attention on $|\<Psi>|^2$. We need to measure the regularity of the stochastic convolution that is formally defined for every $t \in [0,T]$ by:
$$\<Psi>(t,.)=-i\int_{0}^t \frac{e^{-i(t-t')\Delta}}{\langle \nabla \rangle^{\alpha}}\dot{W}(t',.)dt'.$$
In fact, by considering a sequence of smooth processes $(\<Psi>_n)_{n \in \mathbb{N}}$ (see Section \ref{stochastic constr}), we will be able to propose a rigorous construction of $\<Psi>$ and to prove the following result:
\begin{proposition}\label{luxo1}
Let $d \geq 1$ be a space dimension and $T>0$ a positive time. Fix $\alpha$ a positive number and $$s >\frac{d}{2}-\alpha.$$ 
Let $\rho: \mathbb{R}^d \rightarrow \mathbb{R}$ be a $\mathcal{C}^{\infty}$ compactly-supported function and $2 \leq p <\infty$. Then, the sequence $(\rho\<Psi>_n)_{n \in \mathbb{N}}$ converges almost surely in the space $\mathcal{C}([0,T];\mathcal{W}^{-s,p}(\mathbb{R}^d))$.\\
Denoting by $\rho\<Psi>$ the almost sure limit, it holds that
$$\rho\<Psi> \in \mathcal{C}([0,T];\mathcal{W}^{-s,p}(\mathbb{R}^d)).$$
\end{proposition}
The latter proposition sheds light on an important feature. Two regimes have to be distinguished. When $\alpha \geq \frac{d}{2}$, $\<Psi>$ is a function of time with values in a space of functions (up to multiplication by $\rho$) whereas, when $\alpha<\frac{d}{2}$, $\<Psi>$ is a function of time taking values in a negative order Sobolev space. In the first case, for all $0 \leq t \leq T$, $|\<Psi>(t,.)|^2$ makes sense as product of two functions. But, in the second one, we cannot define $|\<Psi>|^2$ as a product of distributions of negative order. In the following, we will only deal with the case where $\alpha<\frac{d}{2}$ that is the hardest one. \\
In order to define $|\<Psi>|^2$, a first idea would be to use a sequential argument and to see $|\<Psi>|^2$ as the limit of the stochastic processes $(|\<Psi>_n|^2)_{n \in \mathbb{N}}$ in a convenient space. But, a quick computation shows that, for all $x \in \mathbb{R}^d$ and $t \in (0,T]$,
$$\mathbb{E}\Big[\big|\<Psi>_n(t,x)\big|^2\Big]=\frac{t}{(2\pi)^d}\int_{|\xi| \leq n}\frac{1}{(1+|\xi|^2)^{\alpha}}d\xi \underset{n \rightarrow + \infty}{\sim}C_{\alpha} t n^{d-2\alpha},$$
that tends to $+ \infty$ as $n$ goes to infinity, preventing us from applying our strategy. In fact, as in \cite{Schaeffer1}, we will proceed with a Wick renormalization and rather consider for every $x \in \mathbb{R}^d$ and $0 \leq t \leq T$,
$$\<Psi2>_n(t,x)=|\<Psi>_n(t,x)|^2-\mathbb{E}\Big[\big|\<Psi>_n(t,x)\big|^2\Big].$$

\begin{remark}
The  value of the renormalization constant $\mathbb{E}\Big[\big|\<Psi>_n(t,x)\big|^2\Big]$ does not depend on $x \in \mathbb{R}^d$, an already underlined point in \cite{Schaeffer1} in the case of the fractional (in time and in space) noise.
\end{remark}

We are now in a position to define our second order stochastic process $\<Psi2>$.
\begin{proposition}\label{cerise1}
Let $d \geq 1$ be a space dimension and $T>0$ a positive time. Fix $\alpha$ and $s$ two positive numbers verifying $$\displaystyle\frac{d}{4}<\alpha<\frac{d}{2}\quad \text{and} \quad s>\frac{d}{2}-\alpha.$$ 
Let $\rho: \mathbb{R}^d \rightarrow \mathbb{R}$ be a $\mathcal{C}^{\infty}$ compactly-supported function and $2 \leq p <\infty$. Then, the sequence $(\rho\<Psi2>_n)_{n \in \mathbb{N}}$ converges almost surely in the space $\mathcal{C}([0,T];\mathcal{W}^{-2s,p}(\mathbb{R}^d))$.\\
Denoting by $\rho \<Psi2>$ the almost sure limit, it holds that $$\rho\<Psi2> \in \mathcal{C}([0,T];\mathcal{W}^{-2s,p}(\mathbb{R}^d)).$$
\end{proposition}
Replacing $|\<Psi>|^2$ by $\<Psi2>$, we rewrite equation (\ref{equation sur v}) under the form
\begin{multline}\label{egalite inter}
v_t=e^{-it\Delta}\phi-i\int_0^t e^{-i(t-\tau)\Delta}(\rho^2 |v_{\tau}|^2)\, d\tau-i\int_0^t e^{-i(t-\tau)\Delta}((\rho \overline{v}_{\tau})\cdot(\rho \<Psi>_{\tau}))\, d\tau\\
-i\int_0^t e^{-i(t-\tau)\Delta}((\rho v_{\tau})\cdot(\overline{\rho \<Psi>_{\tau}}))\, d\tau-i\int_0^t e^{-i(t-\tau)\Delta}(\rho^2 \<Psi2>_{\tau})\, d\tau\, , \quad t\in [0,T] \, .
\end{multline}
We can now come back on the definitions of the products $\rho v \cdot \overline{\rho \<Psi>}$ and $\rho\overline{v} \cdot \rho\<Psi>$. Having a look on equation (\ref{egalite inter}), we see that $v$ is expected to live in $\mathcal{C}([0,T];H^{-2s}(\mathbb{R}^d))$, inheriting the bad regularity of $\rho^2\<Psi2>$. Again, we cannot define a product between two distributions of negative order. In fact, as $\rho: \mathbb{R}^d \rightarrow \mathbb{R}$ is a test function, we can benefit from a local smoothing effect of the Schr\"{o}dinger operator described in Lemma \ref{local-regu}. To be more precise, the latter lemma imposes an additionnal assumption on $\rho$, that is to take $\rho$ of the form
\begin{equation}\label{rho}\tag{$\mathbf{F_\rho}$}
\rho(x_1,\dots,x_d)=\rho_1(x_1)\cdots \rho_d(x_d)
\end{equation}
for smooth compactly-supported functions $\rho_1,\ldots,\rho_d$ on $\R$.
Thanks to this smoothing, up to multiplication by $\rho$, for all $0\leq t \leq T$, $v(t,.)$ lives in $\mathcal{W}^{-2s+\eta,p}(\mathbb{R}^d)$ for small $\eta>0$ such that $-2s+\eta>0$ and where $p \geq 2$. Applying the classical rule (see Lemma \ref{product}) that states that we can define the product of two distributions of Sobolev regularities $-s<0$ and $\beta>0$ as a distribution of Sobolev regularity $-s$ as soon as $\beta>s$, we see that the two terms $\rho v \cdot \overline{\rho\<Psi>}$ and $\rho\overline{v} \cdot \rho\<Psi>$ make sense. And, thanks to the local smoothing again, the quadratic term $\rho^2 |v|^2$ can be interpreted as a product of two functions. Finally, each term in the expression of (\ref{egalite inter}) is well-defined. We can now propose our interpretation of the model $(\ref{dim-d-sch})$.

\subsection{Interpretation and local well-posedness of equation \eqref{dim-d-sch}}\label{pointfixe1}

\begin{definition}\label{sol-sch1}
Let $d \geq 1$ be a space dimension and $T>0$ a positive time. Fix $\alpha$ and $s$ two real numbers verifying $$\displaystyle\frac{d}{4}<\alpha<\frac{d}{2}\quad \text{and} \quad s>\frac{d}{2}-\alpha.$$ 
A stochastic process $(u(t,x))_{t \in [0,T], x \in \mathbb{R}^d}$ is said to be a Wick solution on $[0,T]$ of equation~\eqref{dim-d-sch} if, almost surely, the process $v:=u-\<Psi>$ is a solution of the mild equation
\begin{multline*}
v_t=e^{-it\Delta}\phi-i\int_0^t e^{-i(t-\tau)\Delta}(\rho^2 |v_{\tau}|^2)\, d\tau-i\int_0^t e^{-i(t-\tau)\Delta}((\rho \overline{v}_{\tau})\cdot(\rho \<Psi>_{\tau}))\, d\tau\\
-i\int_0^t e^{-i(t-\tau)\Delta}((\rho v_{\tau})\cdot(\overline{\rho \<Psi>_{\tau}}))\, d\tau-i\int_0^t e^{-i(t-\tau)\Delta}(\rho^2 \<Psi2>_{\tau})\, d\tau\, , \quad t\in [0,T] \, .
\end{multline*}
\end{definition}

Let us now present our first result. As in \cite{Schaeffer1}, the key point is the introduction of the local Sobolev space below
$$H^{-2s+\eta}_\rho(\R^d):=\{ v \in \mathcal{S}^{'}(\mathbb{R}^d);\ \|\rho \cdot (\id-\Delta)^{\frac{-2s+\eta}{2}}(v)\|_{L^2(\mathbb{R}^d)}<\infty \},$$
for $s \in \mathbb{R}$ and $\eta>0$, that permits us to benefit from the local smoothing of the Schr\"{o}dinger operator. We only present our result under the form of a proposition because we want to insist on the fact that the main novelty of this paper lies in Theorem \ref{major} resulting from the multilinear smoothing developed in the next subsection. 

\begin{proposition}
Let $1\leq d\leq 3$ be a space dimension and $\rho: \mathbb{R}^d \rightarrow \mathbb{R}$ be a $\mathcal{C}^{\infty}$ compactly-supported function of the form \eqref{rho}. Besides, assume that 
\begin{equation*}
\alpha_d<\alpha<\frac{d}{2} \, , \quad \text{where} \ \ \alpha_d=
\left\{
\begin{array}{l}
7/20 \quad \text{if} \ d=1\\
18/20 \quad \text{if} \ d=2\\
29/20 \quad \text{if} \ d=3 
\end{array}
\right. \, .
\end{equation*}
Fix $s>0$ such that $\frac{d}{2}-\alpha<s<s_d$, where
\begin{equation*}
s_d=
\left\{
\begin{array}{l}
3/20 \quad \text{if} \ d=1\\
1/10 \quad \text{if} \ d=2\\
1/24 \quad \text{if} \ d=3 
\end{array}
\right. \, .
\end{equation*}
\noindent
Then the following assertions hold true:

\smallskip

\noindent
$(i)$ One can find parameters $\eta\in [2s,1/2]$ and $p,q\geq 2$ such that, almost surely, for every $\phi \in H^{-2s}(\mathbb{R}^d)$, there exists a time $T_0>0$ for which equation~\eqref{dim-d-sch} admits a unique Wick solution $u$ (in the sense of Definition~\ref{sol-sch1}) in the set
$$\mathcal{S}_{T_0}:= \<Psi> + X^{s,\eta,(p,q)}_\rho(T_0)\, ,$$
where
$$X^{s,\eta,(p,q)}_\rho(T):=\mathcal{C}([0,T]; H^{-2s}(\mathbb{R}^d))\cap L^p([0,T]; \cw^{-2s,q}(\mathbb{R}^d))\cap L^{\frac{1}{\eta}}_T H^{-2s+\eta}_\rho \, .$$

\smallskip

\noindent
$(ii)$ For every $n\geq 1$, let $\widetilde{u}_n$ denote the smooth Wick solution of~\eqref{dim-d-sch}, that is $\widetilde{u}_n$ is the solution (in the sense of Definition~\ref{sol-sch1}) associated with the pair $(\rho \<Psi>_n,\rho^2\<Psi2>_n)$. For all $\mathcal{C}^{\infty}$ compactly-supported functions $\chi: \mathbb{R}^d \rightarrow \mathbb{R}$ , the sequence $(\chi\widetilde{u}_{n})_{n\geq 1}$ converges almost surely in $\mathcal{C}([0,T_0];H^{-2s}(\mathbb{R}^d))$ to $\chi u$, where $u$ is the Wick solution exhibited in item $(i)$.
\end{proposition}

This proposition can be proved with analogous arguments as those developed in \cite{Schaeffer1}. Precisely, the stochastic processes $\<Psi>$ and $\<Psi2>$ being constructed, it results from a deterministic fixed-point argument.

\begin{remark}
In \cite{Schaeffer1}, the authors have studied a similar model, the noise $\langle \nabla \rangle ^{-\alpha}\dot W$ being replaced by the derivative of a fractional (in time and in space) Brownian field $B$. They have not constructed the second order stochastic process $\<Psi2>$ when $B$ is a white noise, that is when $H_0=H_1=\dots=H_d=\frac{1}{2}$. Proposition \ref{cerise1} shows that this construction can be realized under a hypothesis of spatial regularization, namely when $\frac{d}{4}<\alpha< \frac{d}{2}$. In fact, we suspect that this condition on $\alpha$ is optimal and that $\<Psi2>$ could not be defined as a function of time taking values in a space of distributions in the case of the white noise, that is when $\alpha=0$. 
\end{remark}

We can observe that we are able to solve equation (\ref{dim-d-sch}) only for a small range of $\alpha$. This comes from the difficulty to bound $|v|^2$. Let us now follow a different strategy. We have seen that $\rho^2\<Psi2> \in \mathcal{C}([0,T]; H^{-2s}(\mathbb{R}^d))$ and that, consequently, $v$ inherits its bad regularity. That is why we propose a stochastic construction of $\<IPsi2>(t,.)=\displaystyle-i\int_{0}^{t}e^{-i(t-\tau)\Delta}(\<Psi2>(\tau,.))d\tau$ and hope for a better regularity. Let us develop this point in the next subsections.

\subsection{Multilinear smoothing}

In order to construct the convolution between the Schr\"{o}dinger operator and the second order stochastic process $\<Psi2>$, we consider the following sequence:
\begin{definition}\label{cerisessj}
For every $t \in [0,T]$,
$$\<IPsi2>_n(t,.)=-i\int_{0}^{t}e^{-i(t-\tau)\Delta}(\<Psi2>_n(\tau,.))d\tau.$$
\end{definition}
As usual, the strategy is to show that the latter sequence is a Cauchy sequence in a convenient subspace. Our next result reads as follows.
\begin{proposition}\label{cerise-ssj}
Let $1 \leq d \leq 3$ be a space dimension and $T>0$ a positive time. Fix $\alpha$ and $s$ two real numbers verifying $$\displaystyle\frac{d}{4}<\alpha<\frac{d}{2}\quad \text{and} \quad s>\frac{d}{2}-\alpha.$$
Assume that\begin{equation*}\label{kappa}
\kappa=
\left\{
\begin{array}{l}
1-\alpha \quad \text{if} \ d=1\\
\displaystyle\frac{3}{2}-\alpha \quad \text{if} \ d=2\\
2-\alpha \quad \text{if} \ d=3 \ \text{and} \ \alpha\geq 1 \\
1 \quad \text{if} \ d=3 \ \text{and} \ \alpha< 1 \  .
\end{array}
\right.
\end{equation*}
Let $\rho: \mathbb{R}^d \rightarrow \mathbb{R}$ be a $\mathcal{C}^{\infty}$ compactly-supported function and $2 \leq p <\infty$. Then, the sequence $\Big(\rho\<IPsi2>_n\Big)_{n \in \mathbb{N}}$ converges almost surely in the space $\mathcal{C}([0,T];\mathcal{W}^{-2s+\kappa,p}(\mathbb{R}^d))$.
Denoting by $\rho\<IPsi2>$ the almost sure limit, it holds that
$$\rho\<IPsi2> \in \mathcal{C}([0,T];\mathcal{W}^{-2s+\kappa,p}(\mathbb{R}^d)).$$
\end{proposition}

\begin{remark}\label{gain}
Suppose that the assumptions of the latter proposition are verified.\\
$\bullet$ When $d=1$, for all $0\leq t \leq T$, $\rho\<IPsi2>(t,.) \in H^{\alpha-\varepsilon}(\mathbb{R})$ for every $\varepsilon >0$. In particular, if $\varepsilon$ is small enough, $\rho\<IPsi2>$ is a function of time with values in a space of functions. Thanks to this multilinear smoothing, we will be able to settle a fixed-point argument for all the range of $\alpha$, namely when $\frac{1}{4}<\alpha<\frac{1}{2}$.\\
$\bullet$ When $d=2$, for all $0\leq t \leq T$, $\rho\<IPsi2>(t,.) \in H^{\alpha-\frac{1}{2}-\varepsilon}(\mathbb{R}^2)$ for every $\varepsilon >0$. In particular, if $\varepsilon$ is small enough, $\rho\<IPsi2>$ is again a function of time with values in a space of functions. This time, we will be able to settle a fixed-point argument only when $\frac{5}{6}<\alpha<1$. As in \cite{Schaeffer1}, this constraint will result from the deterministic part of our study.\\
$\bullet$ When $d=3$ and $\alpha>1$, for all $0\leq t \leq T$, $\rho\<IPsi2>(t,.) \in H^{\alpha-1-\varepsilon}(\mathbb{R}^3)$ for every $\varepsilon >0$. In particular, if $\varepsilon$ is small enough, $\rho\<IPsi2>$ is a function of time with values in a space of functions but, when $\alpha \leq 1$, despite the multilinear smoothing, $\rho\<IPsi2>$ is a function of time with values in a negativ order Sobolev space. Here, as in dimension $2$, the deterministic part of our study will impose a more restrictive condition on $\alpha$, namely $\frac{17}{12}<\alpha<\frac{3}{2}$.
\end{remark}

\begin{remark}
The multilinear smoothing defined by the real $\kappa>0$ of Proposition \ref{cerise-ssj} is better than the local smoothing of Lemma \ref{local-regu} that only provides a gain of a half of a derivative (understood in the Sobolev meaning). One can observe that the smoothing effect is a decreasing function of $\alpha$. In other words, the more the noise is regular, the less the extra smoothing is strong. This a common point shared with the waves setting. Indeed, when $d=2$, Oh and Okamoto \cite{Oh} proved a gain of $\min(\alpha,\frac{1}{4})$ of a derivative.
\end{remark}

\subsection{A deformed version of equation \eqref{dim-d-sch}}

Let us go back to equation (\ref{egalite inter}). Replacing $\displaystyle-i\int_0^t e^{-i(t-\tau)\Delta}(\rho^2\<Psi2>_{\tau})\, d\tau$ by $\rho^2\<IPsi2>$, we obtain
\begin{multline}\label{egalite inter3}
v_t=e^{-it\Delta}\phi-i\int_0^t e^{-i(t-\tau)\Delta}( \rho^2|v_{\tau}|^2)\, d\tau-i\int_0^t e^{-i(t-\tau)\Delta}( (\rho\overline{v}_{\tau})\cdot (\rho\<Psi>_{\tau}))\, d\tau\\
-i\int_0^t e^{-i(t-\tau)\Delta}( (\rho v_{\tau})\cdot(\overline{\rho \<Psi>_{\tau}}))\, d\tau+\rho^2\<IPsi2>\, , \quad t\in [0,T] \, .
\end{multline}
We underline the fact that replacing $\displaystyle-i\int_0^t e^{-i(t-\tau)\Delta}(\rho^2\<Psi2>_{\tau})\, d\tau$ by $\rho^2\<IPsi2>$ in (\ref{egalite inter}) changes the model (\ref{dim-d-sch}) since we have extracted the test function $\rho^2$ from the integral. Precisely, the previous equality corresponds to the following deformed version of (\ref{dim-d-sch}):
\begin{equation}\label{deformed}
\left\{
\begin{array}{l}
i\partial_t u-\Delta u= \rho^2 |u|^2 + C_{\rho}+\langle \nabla \rangle^{-\alpha}\dot{W} \, , \quad  t\in [0,T] \, , \, x\in \R^d \, ,\\
u_0=\phi\, ,
\end{array}
\right.
\end{equation}
where $C_{\rho}$ is a renormalization \emph{variable} whose expression is given by $$C_{\rho}=(i\partial_t-\Delta)\Big(\rho^2 \<IPsi2>\Big)-\rho^2 \<Psi2>-\rho^2 \mathbb{E}[|\<Psi>|^2].$$
Now we define what we will call a solution of equation (\ref{deformed}):

\begin{definition}\label{sol-sch2}
Let $d \geq 1$ be a space dimension and $T>0$ a positive time. Fix $\alpha$ and $s$ two real numbers verifying $$\displaystyle\frac{d}{4}<\alpha<\frac{d}{2}\quad \text{and} \quad s>\frac{d}{2}-\alpha.$$ 
A stochastic process $(u(t,x))_{t \in [0,T], x \in \mathbb{R}^d}$ is said to be a solution on $[0,T]$ of equation (\ref{deformed}) if, almost surely, the process $v:=u-\<Psi>$ is a solution of the mild equation
\begin{multline}\label{interpretation}
v_t=e^{-it\Delta}\phi-i\int_0^t e^{-i(t-\tau)\Delta}(\rho^2 |v_{\tau}|^2)\, d\tau-i\int_0^t e^{-i(t-\tau)\Delta}((\rho \overline{v}_{\tau})\cdot(\rho \<Psi>_{\tau}))\, d\tau\\
-i\int_0^t e^{-i(t-\tau)\Delta}((\rho v_{\tau})\cdot(\overline{\rho \<Psi>_{\tau}}))\, d\tau+\rho^2\<IPsi2> , \quad t\in [0,T] \, ,
\end{multline}
where $\rho: \mathbb{R}^d \rightarrow \mathbb{R}$ is a $\mathcal{C}^{\infty}$ compactly-supported function.
\end{definition}

Let us study the regularity of each term in equation (\ref{interpretation}). Suppose that $\rho^2\<IPsi2>$ is a function of time with values in a space of functions (that is possible for some range of $\alpha$ according to the multilinear smoothing of Proposition \ref{cerise-ssj}). Then, the term with the worst regularity is $\rho\<Psi>$ and $v$ is expected to live in $\mathcal{C}([0,T];H^{-s}(\mathbb{R}^d))$. Again, there is an issue in making sense of $\rho v \cdot \overline{\rho\<Psi>}$, $\rho\overline{v} \cdot \rho\<Psi>$ and $\rho^2 |v|^2$. In fact, resorting to the local smoothing effect of the Schr\"{o}dinger operator, up to multiplication by $\rho$, for every $0 \leq t \leq T$, $v(t,.)$ lives in $\mathcal{W}^{-s+\eta,p}(\mathbb{R}^d)$ for small $\eta>0$ such that $-s+\eta>0$ and where $p \geq 2$. Consequently, using the same arguments as in Subsection \ref{pointfixe1}, the quadratic term $\rho^2 |v|^2$ can be interpreted as a product of two functions and, under the additional assumption $-s+\eta>s$ allowing us to define $\rho v \cdot \overline{\rho\<Psi>}$ and $\rho\overline{v} \cdot \rho\<Psi>$ as two functions of time with values in a Sobolev space of order $-s$, we see that each term in equation (\ref{interpretation}) is well-defined.\\
We can now state our main local well-posedness result:

\begin{theorem}\label{major}
Let $1\leq d\leq 3$ be a space dimension and $\rho: \mathbb{R}^d \rightarrow \mathbb{R}$ be a $\mathcal{C}^{\infty}$ compactly-supported function of the form \eqref{rho}. Besides, assume that 
\begin{equation*}
\alpha_d<\alpha<\frac{d}{2} \, , \quad \text{where} \ \ \alpha_d=
\left\{
\begin{array}{l}
1/4 \quad \text{if} \ d=1\\
5/6 \quad \text{if} \ d=2\\
17/12 \quad \text{if} \ d=3 
\end{array}
\right. \, .
\end{equation*}

\noindent
Then the following assertions hold true:

\smallskip

\noindent
$(i)$ There exists $\varepsilon>0$ small enough such that, denoting by $s>0$ the real number $s=\frac{d}{2}-\alpha + \varepsilon$, one can find parameters $\eta\in [s,1/2]$ and $p,q\geq 2$ such that, almost surely, for every $\phi \in~H^{-s}(\mathbb{R}^d)$, there exists a time $T_0>0$ for which equation (\ref{deformed}) admits a unique solution $u$ (in the sense of Definition~\ref{sol-sch2}) in the set
$$\mathcal{S}_{T_0}:= \<Psi> + Y^{s,\eta,(p,q)}_\rho(T_0)\, ,$$
where
$$Y^{s,\eta,(p,q)}_\rho(T):=\mathcal{C}([0,T]; H^{-s}(\mathbb{R}^d))\cap L^p([0,T]; \cw^{-s,q}(\mathbb{R}^d))\cap L^{\frac{1}{\eta}}_T H^{-s+\eta}_\rho \, .$$

\smallskip

\noindent
$(ii)$ For every $n\geq 1$, let $\widetilde{u}_n$ denote the smooth solution of (\ref{deformed}), that is $\widetilde{u}_n$ is the solution (in the sense of Definition~\ref{sol-sch2}) associated with $(\rho \<Psi>_n,\rho^2\<IPsi2>_n)$. For all $\mathcal{C}^{\infty}$ compactly-supported functions $\chi: \mathbb{R}^d \rightarrow \mathbb{R}$ , the sequence $(\chi\widetilde{u}_{n})_{n\geq 1}$ converges almost surely in $\mathcal{C}([0,T_0];H^{-s}(\mathbb{R}^d))$ to $\chi u$, where $u$ is the solution exhibited in item $(i)$.
\end{theorem}

\begin{remark}
The proof of this theorem results from a deterministic fixed-point argument. A precise description of the different tools we will resort to is given in Section \ref{point fixe}.
\end{remark}

\begin{remark}
Thanks to the multilinear smoothing, we see that, in dimension one, we are able to solve our model for all $\frac{1}{4}<\alpha<\frac{1}{2}$ (that is as soon as we are able to construct $\<IPsi2>$) and that, when $d=2$ or $d=3$, an additional constraint on $\alpha$ appears. This condition comes from the deterministic part of our work. See Section \ref{point fixe} for more details.
\end{remark}

\begin{remark}
In all this study, the $\mathcal{C}^{\infty}$ compactly-supported function $\rho$ plays a major role. Formally, when $\rho=1$, Definition \ref{sol-sch2} is equivalent to the mild form of equation (\ref{dim-d-sch}) (after renormalization). An important step would be to establish the local well-posedness of (\ref{dim-d-sch}) without any cut-off but, insofar as the Schr\"{o}dinger operator does not provide any form of global regularization, it currently lies beyond our ken.
\end{remark}

\subsection{Notations}

Let $d\geq 1$ be a space dimension. Let us recall some classical notations.\\
\noindent
$\bullet$ $\mathcal{S}(\R^d)$ is the space of Schwartz functions on $\R^d$.\\
$\bullet$ For all $s\in \mathbb{R}$ and $1\leq p< \infty$,
$$\mathcal{W}^{s,p}(\mathbb{R}^d)=\left\{f\in \mathcal{S}'(\mathbb{R}^d):\ \|f \|_{\mathcal{W}^{s,p}}=\|\mathcal{F}^{-1}(\{1+|.|^2\}^{\frac{s}{2}}\mathcal{F}f) |L^p(\mathbb{R}^d)\| 
  <\infty \right\} \, ,$$
where the Fourier transform $\cf$ and the inverse Fourier transform $\cf^{-1}$ are defined by the following formula: for all $f \in \mathcal{S}(\mathbb{R}^d)$ and $x \in \mathbb{R}^d,$
$$\mathcal{F}(f)(x)=\hat{f}(x)=\int_{\mathbb{R}^d} f(y)e^{-i\langle x, y \rangle}dy\,  \quad \text{and} \ \ \mathcal{F}^{-1}(f)(x)=\frac{1}{(2\pi)^d}\int_{\mathbb{R}^d} f(y)e^{i\langle x, y \rangle}dy\, .$$
$\bullet $ For every $s \in \mathbb{R}$, $H^s(\mathbb{R}^d)=\mathcal{W}^{s,2}(\mathbb{R}^d)$.\\
$\bullet$ The operator $\langle \nabla \rangle^{-\alpha}$ is defined through the Fourier transform formula:
$$\cf\big(\langle \nabla \rangle^{-\alpha}(v)\big)(\xi)=\{1+|\xi|^2\}^{-\frac{\alpha}{2}} \cf v (\xi) \, .$$

\subsection{Organization of the paper}
In section \ref{stochastic constr}, we prove Propositions $\ref{luxo1}$, $\ref{cerise1}$ and $\ref{cerise-ssj}$. Then, in Section $\ref{point fixe}$, we establish the deterministic fixed-point leading to the local well-posedness of our model.

\section{On the construction of the relevant stochastic objects}\label{stochastic constr}

In this section, we propose a rigorous construction of the three stochastic processes $\<Psi>$, $\<Psi2>$ and $\<IPsi2>$ at the core of our problem. The first order stochastic process $\<Psi>$ is solution of
\begin{equation*}
\left\{
\begin{array}{l}
i\partial_t \<Psi>-\Delta \<Psi>= \langle \nabla \rangle^{-\alpha}\dot{W} \, , \quad  t\in [0,T] \, , \, x\in \R^d \, ,\\
\<Psi>(0,.)=0\, .
\end{array}
\right.
\end{equation*}
Formally, it is defined for every $t \in [0,T]$ by:
$$\<Psi>(t,.)=-i\int_{0}^t \frac{e^{-i(t-t')\Delta}}{\langle \nabla \rangle^{\alpha}}\dot{W}(t',.)dt'.$$
Let us introduce the sequence of truncated stochastic processes $(\<Psi>_n)_{n \in \mathbb{N}}$ defined for all $t \in [0,T]$ by:
$$\<Psi>_n(t,.)=-i\int_{0}^t \frac{e^{-i(t-t')\Delta}}{\langle \nabla \rangle^{\alpha}}\chi_n(\nabla)\dot{W}(t^{\prime},.)dt',$$
where $\chi_n$ is the indicator function of the ball of radius $n$ in $\mathbb{R}^d$. One may observe that an analogous regularization procedure has been used in the waves setting by Tolomeo (see \cite{Tolomeo}).
Let us consider, for every $n \in \mathbb{N}$, the operator $A_n$ defined for any $t \in [0,T]$ by:
$$A_n(t)=\frac{e^{-it\Delta}}{\langle \nabla \rangle^{\alpha}}\chi_n(\nabla).$$
It holds that, for all $\phi \in \mathcal{S}(\mathbb{R}^d)$,
\begin{eqnarray*}
[A_n(t)\phi](x)&=&\mathcal{F}^{-1}\bigg(\frac{e^{it|\xi|^2}}{(1+|\xi|^2)^{\frac{\alpha}{2}} }\chi_n(|\xi|)\mathcal{F}(\phi)\bigg)(x)\\
&=&\mathcal{F}^{-1}\bigg(\frac{e^{it|\xi|^2}}{(1+|\xi|^2)^{\frac{\alpha}{2}} }\chi_n(|\xi|)\bigg)\star\phi(x)\\
&=&a_n(t,.) \star \phi (x),
\end{eqnarray*}
where $$a_n(t,x)=\frac{1}{(2\pi)^d}\int_{B_n}\frac{e^{it|\xi|^2}}{(1+|\xi|^2)^{\frac{\alpha}{2}} }e^{i\langle x,\xi \rangle}d\xi.$$
Consequently, $\<Psi>_n$ can be properly defined in the following way: for all $(t,x) \in [0,T]\times \mathbb{R}^d$,
\begin{align}\<Psi>_n(t,x)=-i\int_0^t \int_{\mathbb{R}^d}a_{n}(t-t',x-x')\dot{W}(t',x')dt'dx'=-i\int_0^t \int_{\mathbb{R}^d}a_{n}(t-t',x-x')W(dt',dx'),\label{defiluxo}
\end{align}
and Ito's isometry allows us to compute the associated covariance function.

\begin{definition}\label{luxn}
Let $d \geq 1$ be a space dimension and $T>0$ a positive time. We call regularized stochastic convolution any centered complex Gaussian process
$$\Big\{\<Psi>_n(s,x), n \in \mathbb{N}, 0\leq s \leq T, x \in \mathbb{R}^d\Big\}$$
whose covariance function verifies: for all $(n,m) \in \mathbb{N}^2$, $(s,t) \in [0,T]^2$ and $(x,y)\in \mathbb{R}^d$,
\begin{align*}
&\mathbb{E}\Big[\<Psi>_n(s,x)\overline{\<Psi>_m(t,y)}\Big]=\min(s,t)\frac{1}{(2\pi)^d}\int_{B_n \cap B_m}\frac{e^{i(s-t)|\xi|^2}}{(1+|\xi|^2)^{\alpha}}e^{-i\langle \xi, x - y\rangle}d\xi,\\
&\mathbb{E}\Big[\<Psi>_n(s,x)\<Psi>_m(t,y)\Big]=\frac{1}{(2\pi)^d}\int_{0}^{\min(s,t)}\int_{B_n \cap B_m}\frac{e^{i(s+t-2t')|\xi|^2}}{(1+|\xi|^2)^{\alpha}}e^{-i\langle \xi, x - y\rangle}d\xi dt',
\end{align*}
where $B_n=\Big\{\xi \in \mathbb{R}^d, |\xi|\leq n \Big\}$.
\end{definition}

\subsection{Technical lemmas}
We begin by introducing our stochastic tools. The first one is Wick's formula that assures that the mean of a product of Gaussian random variables can be written as a sum of product of means of only two Gaussian random variables. The interested reader shall find more details in \cite{wick-ter} for instance.
\begin{lemma}\label{wick formula}
Let $(\Omega, \mathcal{F},\mathbb{P})$ be a probability space. Let $n \geq 1$ and $X_1,\dots,X_{2n}$ be real-valued jointly Gaussian random variables. Then, it holds that
$$\mathbb{E}[X_1\cdots X_{2n}]=\sum_{\text{pairings}\hspace{0,1cm}\mathcal{P}\hspace{0,1cm}\text{of}\hspace{0,1cm}\{1,\dots,2n\}} \prod_{(i,j)\in \mathcal{P}} \mathbb{E}[X_i X_j]$$
where a pairing of $\{1,\dots,2n\}$ is a partition $\mathcal{P}=\{\{i_1,j_1\},\dots,\{i_n,j_n\}\}$ of this set into disjoint subsets of two elements.
\end{lemma}
\noindent
The lemma below describes the hypercontractivity of Wiener Chaoses and deals with products of Gaussian random variables. It will permit us to turn $L^2(\Omega)$-bounds into $L^{p}(\Omega)$-bounds for $p\geq 2$. A classical reference on this topic is \cite{wick-ter}.
\begin{lemma}\label{chaos}
Let $(\Omega, \mathcal{F},\mathbb{P})$ be a probability space. Let $k \geq 1$ and $c(n_1,\dots,n_k) \in \mathbb{C}$. Let $d \geq 1$ and $g_1,\dots,g_d$ a sequence of independent standard complex-valued Gaussian random variables. Let $S_k: \Omega \rightarrow \mathbb{C}$ the random variable defined by, for all $\omega \in \Omega$,
$$S_k(\omega)=\sum_{\Gamma(k,d)}c(n_1,\dots,n_k)g_{n_1}(\omega)\cdots g_{n_k}(\omega),$$
where $$\Gamma(k,d)=\{(n_1,\dots,n_k) \in \{1,\dots,d\}^k\}.$$
Then, for every $p \geq 2$, it holds that
$$\|S_k\|_{L^p(\Omega)}\leq \sqrt{k+1}(p-1)^{\frac{k}{2}}\|S_k\|_{L^2(\Omega)}.$$
\end{lemma}
\noindent
Let us now present our two main deterministic tools. The following lemma will be of constant use to bring back computations on compact domains. In particular, we will be able to construct our stochastic processes in $\mathcal{W}^{-s,p}(\mathbb{R}^d)$ for some $s \in \mathbb{R}$ and for all $2\leq p <\infty$, up to multiplication by a test function. A proof of this lemma can be found in \cite{Schaeffer1}.
\begin{lemma}\label{compact}
Let $\alpha \in \mathbb{R}$ and $\rho: \mathbb{R}^d \rightarrow \mathbb{R}$ a $\mathcal{C}^{\infty}$ compactly-supported function. For all $p \geq 1$ and $(\eta_1,\dots,\eta_p) \in (\mathbb{R}^d)^p$, it holds that
$$\Bigg|\int_{\mathbb{R}^d}dx \prod_{i=1}^p \iint_{(\mathbb{R}^d)^2}\frac{d\lambda_i d\lati_i}{(1+|\lambda_i|^2)^{\frac{\alpha}{2}}(1+|\lati_i|^2)^{\frac{\alpha}{2}}}e^{i \langle x, \lambda_i - \lati_i \rangle}\hat{\rho}(\lambda_i-\eta_i)\overline{\hat{\rho}(\lati_i-\eta_i)}\Bigg|\lesssim \prod_{i=1}^{p}\frac{1}{(1+|\eta_i|^2)^{\alpha}},$$
where the proportional constant only depends on $\rho$ and $\alpha$.
\end{lemma}
\noindent
Our last lemma will allow us to control some continuous convolutions in the proof of $\<Psi2>$ and $\<IPsi2>$ constructions. A proof can be found for instance in \cite{Ginibre} for $d=1$ or in \cite[Lemma 4.1 p 26]{Mourrat} for all $d \geq 1$ in the discrete case (it suffices to replace the summation by an integral to obtain the result below).
\begin{lemma}\label{convolution}
Let $d \geq 1$ be a space dimension and $\alpha$ and $\beta$ be two non-negative real numbers. If
$$\alpha < d, \quad \beta < d \quad \text{and} \quad \alpha + \beta >d \, ,$$
then, for every $x \in \mathbb{R}^d$, the following bound holds true
$$\int_{\mathbb{R}^d}\frac{dy}{(1+|x-y|^2)^{\frac{\alpha}{2}}}\frac{1}{(1+|y|^2)^{\frac{\beta}{2}}}\lesssim \frac{1}{(1+|x|^2)^{\frac{\alpha+\beta-d}{2}}}.$$
\end{lemma}

\subsection{Construction of the first order stochastic process}

\begin{proof}[Proof of Proposition~\ref{luxo1}]
${}$\\
Let $\rho: \mathbb{R}^d \rightarrow \mathbb{R}$ be a $\mathcal{C}^{\infty}$ compactly-supported function and $2 \leq p <\infty$. Let $s$ and $t$ be two real numbers such that $0\leq s \leq t \leq T$. In the following, for all $(m,n) \in \mathbb{N}^2$, we will use the shortcut notations: $\<Psi>_{m,n}=\<Psi>_m - \<Psi>_n$, $a_{m,n}=a_m - a_n$ and $B_{m,n}=B_m \backslash B_n$. As in \cite{Schaeffer1, Schaeffer2}, our aim is to prove that for all $(m,n) \in \mathbb{N}^2$ such that $1 \leq n \leq m$, the following bound holds true:
$$\int_{\mathbb{R}^d}dx\,\mathbb{E}\bigg[\Big|\mathcal{F}^{-1}\Big(\big(1+|.|^2\big)^{-\frac{s}{2}}\mathcal{F}\Big(\rho [\<Psi>_{m,n}(t,.)-\<Psi>_{m,n}(s,.)]\Big)\Big)(x)\Big|^{2p}\bigg]\lesssim \frac{(t-s)^{\varepsilon p}}{n^{2\varepsilon p}},$$
for $\varepsilon>0$ small enough.\\
First of all, for every $x \in \mathbb{R}^d$, we write
\begin{align*}
&\mathbb{E}\bigg[\Big|\mathcal{F}^{-1}\Big(\big(1+|.|^2\big)^{-\frac{s}{2}}\mathcal{F}\Big(\rho [\<Psi>_{m,n}(t,.)-\<Psi>_{m,n}(s,.)]\Big)\Big)(x)\Big|^2\bigg]\\
&=\frac{1}{(2\pi)^{4d}} \iint_{(\mathbb{R}^d)^2}\frac{d\lambda d\lati}{(1+|\lambda|^2)^{\frac{s}{2}}(1+|\lati|^2)^{\frac{s}{2}}}e^{i \langle x, \lambda - \lati \rangle} \iint_{(\mathbb{R}^d)^2} d\beta d\betati \, \hat{\rho}(\lambda-\beta)\overline{\hat{\rho}(\lati-\betati)}\\
&\hspace{0,5cm}\mathbb{E}\bigg[\mathcal{F}\Big(\<Psi>_{m,n}(t,.)-\<Psi>_{m,n}(s,.)\Big)(\beta)\overline{\mathcal{F}\Big(\<Psi>_{m,n}(t,.)-\<Psi>_{m,n}(s,.)\Big)(\betati)}\bigg]\\
&=\frac{1}{(2\pi)^{4d}} \iint_{(\mathbb{R}^d)^2}\frac{d\lambda d\lati}{(1+|\lambda|^2)^{\frac{s}{2}}(1+|\lati|^2)^{\frac{s}{2}}}e^{i \langle x, \lambda - \lati \rangle} \iint_{(\mathbb{R}^d)^2} d\beta d\betati \, \hat{\rho}(\lambda-\beta)\overline{\hat{\rho}(\lati-\betati)}\\
&\hspace{0,5cm}\int_{\mathbb{R}^d}dy \, e^{-i\langle \beta, y \rangle}\int_{\mathbb{R}^d}d\tilde{y}\, e^{i\langle \betati, \tilde{y}\rangle}\mathbb{E}\bigg[\big(\<Psi>_{m,n}(t,y)-\<Psi>_{m,n}(s,y)\big)\big(\overline{\<Psi>_{m,n}(t,\tilde{y})-\<Psi>_{m,n}(s,\tilde{y})}\big)\bigg].
\end{align*}
Let us focus on the estimation of the previous mean. Coming back to the definition of $\<Psi>_n$ (\ref{defiluxo}), we compute
\begin{align*}
&\<Psi>_{m,n}(t,y)-\<Psi>_{m,n}(s,y)=-i\int_0^t \int_{\mathbb{R}^d}a_{m,n}(t-t',y-x')W(dt',dx')+i\int_0^s \int_{\mathbb{R}^d}a_{m,n}(s-t',y-x')W(dt',dx')\\
&=-i\int_s^t \int_{\mathbb{R}^d}a_{m,n}(t-t',y-x')W(dt',dx')-i\int_0^s \int_{\mathbb{R}^d}\Big[a_{m,n}(t-t',y-x')-a_{m,n}(s-t',y-x')\Big]W(dt',dx')
\end{align*}
and, in a similar way,
\begin{align*}
&\overline{\<Psi>_{m,n}(t,\tilde{y})-\<Psi>_{m,n}(s,\tilde{y})}\\
&=i\int_s^t \int_{\mathbb{R}^d}\overline{a_{m,n}(t-t',\tilde{y}-x')}W(dt',dx')+i\int_0^s \int_{\mathbb{R}^d}\Big[\overline{a_{m,n}(t-t',\tilde{y}-x')-a_{m,n}(s-t',\tilde{y}-x')}\Big]W(dt',dx').
\end{align*}
Resorting successively to Ito's isometry and Plancherel's identity, we obtain
\begin{align}
&\mathbb{E}\bigg[\big(\<Psi>_{m,n}(t,y)-\<Psi>_{m,n}(s,y)\big)\big(\overline{\<Psi>_{m,n}(t,\tilde{y})-\<Psi>_{m,n}(s,\tilde{y})}\big)\bigg]=\int_s^t \int_{\mathbb{R}^d}a_{m,n}(t-t',y-x')\overline{a_{m,n}(t-t',\tilde{y}-x')}dt'dx'\nonumber\\
&+\int_0^s \int_{\mathbb{R}^d}\Big[a_{m,n}(t-t',y-x')-a_{m,n}(s-t',y-x')\Big]\Big[\overline{a_{m,n}(t-t',\tilde{y}-x')-a_{m,n}(s-t',\tilde{y}-x')}\Big]dt'dx'\nonumber\\
&=\frac{1}{(2\pi)^d}\int_s^t \int_{B_{m,n}}\frac{1}{(1+|\xi|^2)^{\alpha}}e^{-i\langle \xi, y-\tilde{y}\rangle}dt'd\xi+\frac{1}{(2\pi)^d}\int_0^s \int_{B_{m,n}}\frac{\big|e^{i(t-t')|\xi|^2}-e^{i(s-t')|\xi|^2}\big|^2}{(1+|\xi|^2)^{\alpha}}e^{-i\langle \xi, y-\tilde{y}\rangle}dt'd\xi. \label{transfo}
\end{align}
Now, using equality (\ref{transfo}), it holds that
\begin{align}
&\int_{\mathbb{R}^d}dy \, e^{-i\langle \beta, y \rangle}\int_{\mathbb{R}^d}d\tilde{y}\, e^{i\langle \betati, \tilde{y}\rangle}\mathbb{E}\bigg[\big(\<Psi>_{m,n}(t,y)-\<Psi>_{m,n}(s,y)\big)\big(\overline{\<Psi>_{m,n}(t,\tilde{y})-\<Psi>_{m,n}(s,\tilde{y})}\big)\bigg]\nonumber\\
&=(2\pi)^d \, (t-s)\int_{B_{m,n}}\delta_{\{\beta=\xi\}}\delta_{\{\betati=\xi\}}\frac{d\xi}{(1+|\xi|^2)^{\alpha}}+(2\pi)^d \, \int_0^s \int_{B_{m,n}}\delta_{\{\beta=\xi\}}\delta_{\{\betati=\xi\}} \frac{\big|e^{i(t-t')|\xi|^2}-e^{i(s-t')|\xi|^2}\big|^2}{(1+|\xi|^2)^{\alpha}}dt'd\xi \label{controle1}
\end{align}
that leads us to
\begin{align*}
&\mathbb{E}\bigg[\Big|\mathcal{F}^{-1}\Big(\big(1+|.|^2\big)^{-\frac{s}{2}}\mathcal{F}\Big(\rho [\<Psi>_{m,n}(t,.)-\<Psi>_{m,n}(s,.)]\Big)\Big)(x)\Big|^2\bigg]\\
&=\frac{1}{(2\pi)^{3d}} \iint_{(\mathbb{R}^d)^2}\frac{d\lambda d\lati}{(1+|\lambda|^2)^{\frac{s}{2}}(1+|\lati|^2)^{\frac{s}{2}}}e^{i \langle x, \lambda - \lati \rangle}\\
&\hspace{0,5cm}\Bigg[(t-s)\int_{B_{m,n}}\frac{\hat{\rho}(\lambda-\xi)\overline{\hat{\rho}(\lati-\xi)}}{(1+|\xi|^2)^{\alpha}}d\xi+\int_0^s \int_{B_{m,n}}\hat{\rho}(\lambda-\xi)\overline{\hat{\rho}(\lati-\xi)}\frac{\big|e^{i(t-t')|\xi|^2}-e^{i(s-t')|\xi|^2}\big|^2}{(1+|\xi|^2)^{\alpha}}dt'd\xi\Bigg].
\end{align*}
Using the hypercontractivity of Gaussian variables, we deduce that
\begin{align*}
&\mathbb{E}\bigg[\Big|\mathcal{F}^{-1}\Big(\big(1+|.|^2\big)^{-\frac{s}{2}}\mathcal{F}\Big(\rho [\<Psi>_{m,n}(t,.)-\<Psi>_{m,n}(s,.)]\Big)\Big)(x)\Big|^{2p}\bigg]\\
&\leq c_p \mathbb{E}\bigg[\Big|\mathcal{F}^{-1}\Big(\big(1+|.|^2\big)^{-\frac{s}{2}}\mathcal{F}\Big(\rho [\<Psi>_{m,n}(t,.)-\<Psi>_{m,n}(s,.)]\Big)\Big)(x)\Big|^2\bigg]^p
\end{align*}
and summoning Lemma \ref{compact}, we get the following estimate
\begin{align}
&\int_{\mathbb{R}^d}dx\,\mathbb{E}\bigg[\Big|\mathcal{F}^{-1}\Big(\big(1+|.|^2\big)^{-\frac{s}{2}}\mathcal{F}\Big(\rho [\<Psi>_{m,n}(t,.)-\<Psi>_{m,n}(s,.)]\Big)\Big)(x)\Big|^{2p}\bigg]\nonumber\\ 
&\lesssim \Bigg(\int_{B_{m,n}}\frac{d\xi}{(1+|\xi|^2)^{s}}\Bigg[\frac{t-s}{(1+|\xi|^2)^{\alpha}}+\int_0^s \frac{\big|e^{i(t-t')|\xi|^2}-e^{i(s-t')|\xi|^2}\big|^2}{(1+|\xi|^2)^{\alpha}}dt'\Bigg]\Bigg)^p=I_{m,n}^p.\label{Imnp}
\end{align}
Let $0<\varepsilon<1$. A straight application of the mean value theorem entails that
\begin{align*}
&I_{m,n}\lesssim \int_{B_{m,n}}\frac{d\xi}{(1+|\xi|^2)^{s}}\Bigg[\frac{t-s}{(1+|\xi|^2)^{\alpha}}+\int_0^s\frac{(t-s)^{\varepsilon}}{(1+|\xi|^2)^{\alpha-\varepsilon}}dt'\Bigg]\\
&\hspace{0,6cm}\lesssim (t-s)^{\varepsilon} \int_{B_{m,n}}\frac{d\xi}{(1+|\xi|^2)^{s+\alpha-\varepsilon}}\\
&\hspace{0,6cm}\lesssim \frac{(t-s)^{\varepsilon}}{n^{2\varepsilon}}\Bigg[1+\int_1^{+\infty}\frac{dr}{r^{2s+2\alpha-4\varepsilon-d+1}}\Bigg],
\end{align*}
where we have performed a hyperspherical change of variables to get the last inequality.
As $s>\frac{d}{2}-\alpha$, we can pick $\varepsilon>0$ small enough so that $2s+2\alpha-4\varepsilon-d+1>1$ and
\begin{align}
I_{m,n}\lesssim \frac{(t-s)^{\varepsilon}}{n^{2\varepsilon}}.\label{Imn}
\end{align}
Injecting (\ref{Imn}) into (\ref{Imnp}), we deduce the desired bound
$$\int_{\mathbb{R}^d}dx\,\mathbb{E}\bigg[\Big|\mathcal{F}^{-1}\Big(\big(1+|.|^2\big)^{-\frac{s}{2}}\mathcal{F}\Big(\rho [\<Psi>_{m,n}(t,.)-\<Psi>_{m,n}(s,.)]\Big)\Big)(x)\Big|^{2p}\bigg]\lesssim \frac{(t-s)^{\varepsilon p}}{n^{2\varepsilon p}},$$
for $\varepsilon>0$ small enough.\\
The end of the proof is classical and is the result of a combination of Kolmogorov's criterion and of the Garsia-Rodemich-Rumsey lemma (see \cite{GRR-ter}). The interested reader shall find the details in \cite{Schaeffer1, Schaeffer2}.
\end{proof}

\subsection{Construction of the second order stochastic process}

\begin{proof}[Proof of Proposition~\ref{cerise1}]
Let $\rho: \mathbb{R}^d \rightarrow \mathbb{R}$ be a $\mathcal{C}^{\infty}$ compactly-supported function and $2 \leq p <\infty$. Let $s$ and $t$ be two real numbers such that $0\leq s \leq t \leq T$. As in \cite{Schaeffer1, Schaeffer2}, our strategy is to obtain a bound of the form:
$$\int_{\mathbb{R}^d}dx\,\mathbb{E}\bigg[\Big|\mathcal{F}^{-1}\Big(\big(1+|.|^2\big)^{-s}\mathcal{F}\Big(\rho [\<Psi2>_{n}(t,.)-\<Psi2>_{n}(s,.)]\Big)\Big)(x)\Big|^{2p}\bigg]\lesssim (t-s)^{\varepsilon p},$$
where $\varepsilon>0$ is small enough and uniformly in $n$. Indeed, a small adaptation of the proof leads to similar results concerning the variation $\<Psi2>_{m,n}=\<Psi2>_m - \<Psi2>_n$, namely
$$\int_{\mathbb{R}^d}dx\,\mathbb{E}\bigg[\Big|\mathcal{F}^{-1}\Big(\big(1+|.|^2\big)^{-s}\mathcal{F}\Big(\rho [\<Psi2>_{m,n}(t,.)-\<Psi2>_{m,n}(s,.)]\Big)\Big)(x)\Big|^{2p}\bigg]\lesssim \frac{(t-s)^{\varepsilon p}}{n^{2\varepsilon p}}.$$
First of all, for every $x \in \mathbb{R}^d$, we write
\begin{align*}
&\mathbb{E}\bigg[\Big|\mathcal{F}^{-1}\Big(\big(1+|.|^2\big)^{-s}\mathcal{F}\Big(\rho [\<Psi2>_{n}(t,.)-\<Psi2>_{n}(s,.)]\Big)\Big)(x)\Big|^2\bigg]\\
&=\frac{1}{(2\pi)^{4d}} \iint_{(\mathbb{R}^d)^2}\frac{d\lambda d\lati}{(1+|\lambda|^2)^{s}(1+|\lati|^2)^{s}} e^{i \langle x, \lambda - \lati \rangle} \iint_{(\mathbb{R}^d)^2} d\beta d\betati \, \hat{\rho}(\lambda-\beta)\overline{\hat{\rho}(\lati-\betati)}\\
&\hspace{0,5cm}\mathbb{E}\bigg[\mathcal{F}\Big(\<Psi2>_{n}(t,.)-\<Psi2>_{n}(s,.)\Big)(\beta)\overline{\mathcal{F}\Big(\<Psi2>_{n}(t,.)-\<Psi2>_{n}(s,.)\Big)(\betati)}\bigg].
\end{align*}
We focus on the estimation of the latter mean.
\begin{align*}
&\mathbb{E}\bigg[\mathcal{F}\Big(\<Psi2>_{n}(t,.)-\<Psi2>_{n}(s,.)\Big)(\beta)\overline{\mathcal{F}\Big(\<Psi2>_{n}(t,.)-\<Psi2>_{n}(s,.)\Big)(\betati)}\bigg]\\
&=\int_{\mathbb{R}^d}dy \, e^{-i\langle \beta, y \rangle}\int_{\mathbb{R}^d}d\tilde{y}\, e^{i\langle \betati, \tilde{y}\rangle}\mathbb{E}\bigg[\big(\<Psi2>_{n}(t,y)-\<Psi2>_{n}(s,y)\big)\big(\overline{\<Psi2>_{n}(t,\tilde{y})-\<Psi2>_{n}(s,\tilde{y})}\big)\bigg].
\end{align*}
Resorting to Wick's formula (see Lemma \ref{wick formula}), we can expand $\mathbb{E}\bigg[\big(\<Psi2>_{n}(t,y)-\<Psi2>_{n}(s,y)\big)\big(\overline{\<Psi2>_{n}(t,\tilde{y})-\<Psi2>_{n}(s,\tilde{y})}\big)\bigg]$ in the following way
\begin{align}
&\mathbb{E}\bigg[\big(\<Psi2>_{n}(t,y)-\<Psi2>_{n}(s,y)\big)\big(\overline{\<Psi2>_{n}(t,\tilde{y})-\<Psi2>_{n}(s,\tilde{y})}\big)\bigg]=\nonumber\\
&\mathbb{E}\bigg[\big(\<Psi>_{n}(t,y)-\<Psi>_{n}(s,y)\big)\overline{\<Psi>_{n}(t,\tilde{y})}\bigg]\mathbb{E}\bigg[\overline{\<Psi>_{n}(t,y)}\<Psi>_{n}(t,\tilde{y})\bigg]+\mathbb{E}\bigg[\<Psi>_{n}(s,y)\overline{\<Psi>_{n}(t,\tilde{y})}\bigg]\mathbb{E}\bigg[\big(\overline{\<Psi>_{n}(t,y)-\<Psi>_{n}(s,y)}\big)\<Psi>_{n}(t,\tilde{y})\bigg]\nonumber\\
&+\mathbb{E}\bigg[\big(\<Psi>_{n}(s,y)-\<Psi>_{n}(t,y)\big)\overline{\<Psi>_{n}(s,\tilde{y})}\bigg]\mathbb{E}\bigg[\overline{\<Psi>_{n}(t,y)}\<Psi>_{n}(s,\tilde{y})\bigg]+\mathbb{E}\bigg[\<Psi>_{n}(s,y)\overline{\<Psi>_{n}(s,\tilde{y})}\bigg]\mathbb{E}\bigg[\big(\overline{\<Psi>_{n}(s,y)-\<Psi>_{n}(t,y)}\big)\<Psi>_{n}(s,\tilde{y})\bigg]\nonumber\\
&+\mathbb{E}\bigg[\big(\<Psi>_{n}(t,y)-\<Psi>_{n}(s,y)\big)\<Psi>_{n}(t,\tilde{y})\bigg]\mathbb{E}\bigg[\overline{\<Psi>_{n}(t,y)\<Psi>_{n}(t,\tilde{y})}\bigg]+\mathbb{E}\bigg[\<Psi>_{n}(s,y)\<Psi>_{n}(t,\tilde{y})\bigg]\mathbb{E}\bigg[\big(\overline{\<Psi>_{n}(t,y)-\<Psi>_{n}(s,y)\big)\<Psi>_{n}(t,\tilde{y})}\bigg]\nonumber\\
&+\mathbb{E}\bigg[\big(\<Psi>_{n}(s,y)-\<Psi>_{n}(t,y)\big)\<Psi>_{n}(s,\tilde{y})\bigg]\mathbb{E}\bigg[\overline{\<Psi>_{n}(t,y)\<Psi>_{n}(s,\tilde{y})}\bigg]+\mathbb{E}\bigg[\<Psi>_{n}(s,y)\<Psi>_{n}(s,\tilde{y})\bigg]\mathbb{E}\bigg[\big(\overline{\<Psi>_{n}(s,y)-\<Psi>_{n}(t,y)\big)\<Psi>_{n}(s,\tilde{y})}\bigg].\label{controlesp}
\end{align}
Consequently, we have to bound eight integrals. Since their treatments are quite the same, we will only detail the computations related to the first one. Namely, our aim is to control the quantity:
\begin{align}
&\mathbbm{I}=\frac{1}{(2\pi)^{4d}} \iint_{(\mathbb{R}^d)^2}\frac{d\lambda d\lati}{(1+|\lambda|^2)^{s}(1+|\lati|^2)^{s}} e^{i \langle x, \lambda - \lati \rangle} \iint_{(\mathbb{R}^d)^2} d\beta d\betati \, \hat{\rho}(\lambda-\beta)\overline{\hat{\rho}(\lati-\betati)}\nonumber\\
&\hspace{0,5cm}\int_{\mathbb{R}^d}dy \, e^{-i\langle \beta, y \rangle}\int_{\mathbb{R}^d}d\tilde{y}\, e^{i\langle \betati, \tilde{y}\rangle}\mathbb{E}\bigg[\big(\<Psi>_{n}(t,y)-\<Psi>_{n}(s,y)\big)\overline{\<Psi>_{n}(t,\tilde{y})}\bigg]\mathbb{E}\bigg[\overline{\<Psi>_{n}(t,y)}\<Psi>_{n}(t,\tilde{y})\bigg].\label{equality}
\end{align}
Expanding the previous mean thanks to Ito's isometry and Plancherel's identity, we obtain
\begin{align}
&\mathbb{E}\bigg[\big(\<Psi>_{n}(t,y)-\<Psi>_{n}(s,y)\big)\overline{\<Psi>_{n}(t,\tilde{y})}\bigg]\mathbb{E}\bigg[\overline{\<Psi>_{n}(t,y)}\<Psi>_{n}(t,\tilde{y})\bigg]=\Bigg[\int_s^t \int_{\mathbb{R}^d}a_{n}(t-t',y-x')\overline{a_{n}(t-t',\tilde{y}-x')}dt'dx'\nonumber\\
&+\int_0^s \int_{\mathbb{R}^d}\Big[a_{n}(t-t',y-x')-a_{n}(s-t',y-x')\Big]\overline{a_{n}(t-t',\tilde{y}-x')}dt'dx'\Bigg]\nonumber\\
&\hspace{8cm}\times\Bigg[\int_0^t \int_{\mathbb{R}^d}\overline{a_n(t-t',y-x')}a_n(t-t',\tilde{y}-x')dt'dx'\Bigg]\nonumber\\
&=\frac{1}{(2\pi)^{2d}}\Bigg[\int_s^t \int_{B_{n}}\frac{1}{(1+|\xi_1|^2)^{\alpha}}e^{-i\langle \xi_1, y-\tilde{y}\rangle}dt'd\xi_1+\int_0^s \int_{B_{n}}\frac{\big(e^{i(t-t')|\xi_1|^2}-e^{i(s-t')|\xi_1|^2}\big)}{(1+|\xi_1|^2)^{\alpha}}\nonumber\\
&\hspace{5,5cm}e^{-i(t-t')|\xi_1|^2}e^{-i\langle \xi_1, y-\tilde{y}\rangle}dt'd\xi_1\Bigg]\Bigg[\int_0^t \int_{B_{n}}\frac{1}{(1+|\xi_2|^2)^{\alpha}}e^{-i\langle \xi_2, y-\tilde{y}\rangle}dt'd\xi_2\Bigg]\label{Icerise}
\end{align}
and, consequently, injecting equality (\ref{Icerise}) into (\ref{equality}), we get that
\begin{align*}
&\mathbbm{I}=\frac{1}{(2\pi)^{4d}} \iint_{(\mathbb{R}^d)^2}\frac{d\lambda d\lati}{(1+|\lambda|^2)^{s}(1+|\lati|^2)^{s}} e^{i \langle x, \lambda - \lati \rangle}  \, \nonumber\\
&\hspace{1cm}\Bigg[(t-s)t \int_{B_{n}}\int_{B_{n}}\frac{\hat{\rho}(\lambda-(\xi_1+\xi_2))\overline{\hat{\rho}(\lati-(\xi_1+\xi_2))}}{(1+|\xi_1|^2)^{\alpha}(1+|\xi_2|^2)^{\alpha}}d\xi_1 d\xi_2\nonumber\\
&\hspace{1cm}+t\int_0^s \int_{B_{n}}\int_{B_{n}}\hat{\rho}(\lambda-(\xi_1+\xi_2))\overline{\hat{\rho}(\lati-(\xi_1+\xi_2))}\frac{\big(e^{i(t-t')|\xi_1|^2}-e^{i(s-t')|\xi_1|^2}\big)}{(1+|\xi_1|^2)^{\alpha}(1+|\xi_2|^2)^{\alpha}}e^{-i(t-t')|\xi_1|^2}dt'd\xi_1d\xi_2\Bigg].
\end{align*}
\noindent
Combining the hypercontractivity of Wiener chaoses (see Lemma \ref{chaos}) and Lemma \ref{compact}, we get the following estimate
\begin{align*}
\int_{\mathbb{R}^d}dx\,\mathbb{E}\bigg[\Big|\mathcal{F}^{-1}\Big(\big(1+|.|^2\big)^{-s}\mathcal{F}\Big(\rho [\<Psi2>_{n}(t,.)-\<Psi2>_{n}(s,.)]\Big)\Big)(x)\Big|^{2p}\bigg] 
\lesssim \Bigg(\sum_{k=1}^{8} I_k\Bigg)^p,
\end{align*}
where the eight integrals are the one related to the terms in (\ref{controlesp}). As above, let us focus on the treatment of the first term denoted by $I_1$. Let $0<\varepsilon<1$. A straight application of the mean value theorem entails that
\begin{align*}
&I_{1}= \Bigg[(t-s)t \int_{\mathbb{R}^d}\int_{\mathbb{R}^d}\frac{1}{(1+|\xi_1+\xi_2|^2)^{2s}}\frac{d\xi_1 d\xi_2}{(1+|\xi_1|^2)^{\alpha}(1+|\xi_2|^2)^{\alpha}}\nonumber\\
&\hspace{0,6cm}+ t\int_0^s \int_{\mathbb{R}^d}\int_{\mathbb{R}^d}\frac{1}{(1+|\xi_1+\xi_2|^2)^{2s}}\frac{\big|e^{i(t-t')|\xi_1|^2}-e^{i(s-t')|\xi_1|^2}\big|}{(1+|\xi_1|^2)^{\alpha}(1+|\xi_2|^2)^{\alpha}}dt'd\xi_1 d\xi_2\Bigg]\nonumber\\
&\hspace{0,6cm}\lesssim (t-s)^{\varepsilon} \iint_{(\mathbb{R}^d)^2}\frac{d\xi_1 d\xi_2}{(1+|\xi_1+\xi_2|^2)^{2s}}\frac{1}{(1+|\xi_1|^2)^{\alpha-\varepsilon}}\frac{1}{(1+|\xi_2|^2)^{\alpha}}\\
&\hspace{0,6cm}\lesssim (t-s)^{\varepsilon} \int_{\mathbb{R}^d}\frac{d\xi_1 }{(1+|\xi_1|^2)^{2s}}\int_{\mathbb{R}^d}\frac{d\xi_2}{(1+|\xi_1-\xi_2|^2)^{\alpha-\varepsilon}}\frac{1}{(1+|\xi_2|^2)^{\alpha}}\\
&\hspace{0,6cm}\lesssim (t-s)^{\varepsilon} \int_{\mathbb{R}^d}\frac{d\xi_1 }{(1+|\xi_1|^2)^{2s+2\alpha-\frac{d}{2}-\varepsilon}},
\end{align*}
where we have used Lemma \ref{convolution} to derive the last inequality for small $\varepsilon>0$ verifying $4\alpha-2\varepsilon >d$.
Now, resorting to a hyperspherical change of variables,
\begin{align*}
\int_{\mathbb{R}^d}\frac{d\xi_1 }{(1+|\xi_1|^2)^{2s+2\alpha-\frac{d}{2}-\varepsilon}}
\lesssim 1+\int_1^{+\infty}\frac{dr}{r^{4s+4\alpha-2\varepsilon-2d+1}}.
\end{align*}
As $s>\frac{d}{2}-\alpha$, we can pick $\varepsilon>0$ so that $4s+4\alpha-2\varepsilon-2d+1>1$ and $$I_{1}\lesssim (t-s)^{\varepsilon}.$$
The end of the proof is classical and is the result of a combination of Kolmogorov's criterion and of the Garsia-Rodemich-Rumsey lemma (see \cite{GRR-ter}). The interested reader shall find the details in \cite{Schaeffer1, Schaeffer2}.
\end{proof}

\subsection{Construction of the convolution with the second order stochastic process}

\begin{proof}[Proof of Proposition~\ref{cerise-ssj}]
Let $\rho: \mathbb{R}^d \rightarrow \mathbb{R}$ be a $\mathcal{C}^{\infty}$ compactly-supported function and $2 \leq p <\infty$. Let $s$ and $t$ be two real numbers such that $0\leq s \leq t \leq T$. Again, our objective is to obtain a bound of the form:
$$\int_{\mathbb{R}^d}dx\,\mathbb{E}\bigg[\Big|\mathcal{F}^{-1}\Big(\big(1+|.|^2\big)^{-s+\frac{\kappa}{2}}\mathcal{F}\Big(\rho \Big[\<IPsi2>_{n}(t,.)-\<IPsi2>_{n}(s,.)\Big]\Big)\Big)(x)\Big|^{2p}\bigg]\lesssim (t-s)^{\varepsilon p},$$
where $\varepsilon>0$ is small enough and uniformly in $n$. Indeed, a small adaptation of the proof leads to similar results concerning the variation $\<IPsi2>_{m,n}=\<IPsi2>_m - \<IPsi2>_n$, namely
$$\int_{\mathbb{R}^d}dx\,\mathbb{E}\bigg[\Big|\mathcal{F}^{-1}\Big(\big(1+|.|^2\big)^{-s+\frac{\kappa}{2}}\mathcal{F}\Big(\rho \Big[\<IPsi2>_{m,n}(t,.)-\<IPsi2>_{m,n}(s,.)\Big]\Big)\Big)(x)\Big|^{2p}\bigg]\lesssim \frac{(t-s)^{\varepsilon p}}{n^{2\varepsilon p}}.$$
First of all, for every $x \in \mathbb{R}^d$, we write
\begin{align}
&\mathbb{E}\bigg[\Big|\mathcal{F}^{-1}\Big(\big(1+|.|^2\big)^{-s+\frac{\kappa}{2}}\mathcal{F}\Big(\rho \Big[\<IPsi2>_{n}(t,.)-\<IPsi2>_{n}(s,.)\Big]\Big)\Big)(x)\Big|^2\bigg]\nonumber\\
&=\frac{1}{(2\pi)^{4d}} \iint_{(\mathbb{R}^d)^2}\frac{d\lambda d\lati}{(1+|\lambda|^2)^{s-\frac{\kappa}{2}}(1+|\lati|^2)^{s-\frac{\kappa}{2}}} e^{i \langle x, \lambda - \lati \rangle} \iint_{(\mathbb{R}^d)^2} d\beta d\betati \, \hat{\rho}(\lambda-\beta)\overline{\hat{\rho}(\lati-\betati)}\nonumber\\
&\hspace{4,7cm}\mathbb{E}\bigg[\mathcal{F}\Big(\<IPsi2>_{n}(t,.)-\<IPsi2>_{n}(s,.)\Big)(\beta)\overline{\mathcal{F}\Big(\<IPsi2>_{n}(t,.)-\<IPsi2>_{n}(s,.)\Big)(\betati)}\bigg]\nonumber\\
&=\frac{1}{(2\pi)^{4d}} \iint_{(\mathbb{R}^d)^2}\frac{d\lambda d\lati}{(1+|\lambda|^2)^{s-\frac{\kappa}{2}}(1+|\lati|^2)^{s-\frac{\kappa}{2}}} e^{i \langle x, \lambda - \lati \rangle} \iint_{(\mathbb{R}^d)^2} d\beta d\betati \, \hat{\rho}(\lambda-\beta)\overline{\hat{\rho}(\lati-\betati)}\nonumber\\
&\hspace{6,4cm}\mathbb{E}\bigg[\mathcal{F}\Big(\<IPsi2>_{n}(t,.)-\<IPsi2>_{n}(s,.)\Big)(\beta)\overline{\mathcal{F}\Big(\<IPsi2>_{n}(t,.)\Big)(\betati)}\bigg]\nonumber\\
&-\frac{1}{(2\pi)^{4d}} \iint_{(\mathbb{R}^d)^2}\frac{d\lambda d\lati}{(1+|\lambda|^2)^{s-\frac{\kappa}{2}}(1+|\lati|^2)^{s-\frac{\kappa}{2}}} e^{i \langle x, \lambda - \lati \rangle} \iint_{(\mathbb{R}^d)^2} d\beta d\betati \, \hat{\rho}(\lambda-\beta)\overline{\hat{\rho}(\lati-\betati)}\nonumber\\
&\hspace{6,4cm}\mathbb{E}\bigg[\mathcal{F}\Big(\<IPsi2>_{n}(t,.)-\<IPsi2>_{n}(s,.)\Big)(\be)\overline{\mathcal{F}\Big(\<IPsi2>_{n}(s,.)\Big)(\betati)}\bigg].\label{egality1}
\end{align}
Based on Definition \ref{cerisessj} of $\<IPsi2>_n$, it holds that, for every $\beta \in \mathbb{R}^d$,
\begin{align*}
&\mathcal{F}\Big(\<IPsi2>_{n}(t,.)-\<IPsi2>_{n}(s,.)\Big)(\beta)=-i\int_s^t e^{i(t-\tau)|\beta|^2}\mathcal{F}(\<Psi2>_n(\tau,.))(\beta)d\tau  \\
&\hspace{6cm}-i\int_0^s \Big(e^{i(t-\tau)|\beta|^2}-e^{i(s-\tau)|\beta|^2}\Big)\mathcal{F}(\<Psi2>_n(\tau,.))(\beta)d\tau.
\end{align*}
The previous identity entails that
\begin{align}
&\mathbb{E}\bigg[\mathcal{F}\Big(\<IPsi2>_{n}(t,.)-\<IPsi2>_{n}(s,.)\Big)(\beta)\overline{\mathcal{F}\Big(\<IPsi2>_{n}(t,.)\Big)(\betati)}\bigg]-\mathbb{E}\bigg[\mathcal{F}\Big(\<IPsi2>_{n}(t,.)-\<IPsi2>_{n}(s,.)\Big)(\beta)\overline{\mathcal{F}\Big(\<IPsi2>_{n}(s,.)\Big)(\betati)}\bigg]\nonumber\\
&=\int_s^t dt_1 \int_0^t dt_2 \, e^{i(t-t_1)|\beta|^2}e^{-i(t-t_2)|\betati|^2}\iint_{(\mathbb{R}^d)^2} dy d\tilde{y} \, e^{-i \langle \beta, y \rangle}e^{i \langle \betati, \tilde{y} \rangle} \mathbb{E}\Big[\<Psi2>_n(t_1,y)\overline{\<Psi2>_n(t_2,\tilde{y})}\Big]\nonumber\\
&+\int_0^s dt_1 \int_0^t dt_2 \, \Big(e^{i(t-t_1)|\beta|^2}-e^{i(s-t_1)|\beta|^2}\Big)e^{-i(t-t_2)|\betati|^2}\iint_{(\mathbb{R}^d)^2} dy d\tilde{y} \, e^{-i \langle \beta, y \rangle}e^{i \langle \betati, \tilde{y} \rangle} \mathbb{E}\Big[\<Psi2>_n(t_1,y)\overline{\<Psi2>_n(t_2,\tilde{y})}\Big]\nonumber\\
&-\int_s^t dt_1 \int_0^s dt_2 \, e^{i(t-t_1)|\beta|^2}e^{-i(s-t_2)|\betati|^2}\iint_{(\mathbb{R}^d)^2} dy d\tilde{y} \, e^{-i \langle \beta, y \rangle}e^{i \langle \betati, \tilde{y} \rangle} \mathbb{E}\Big[\<Psi2>_n(t_1,y)\overline{\<Psi2>_n(t_2,\tilde{y})}\Big]\nonumber\\
&-\int_0^s dt_1 \int_0^s dt_2 \, \Big(e^{i(t-t_1)|\beta|^2}-e^{i(s-t_1)|\beta|^2}\Big)e^{-i(s-t_2)|\betati|^2}\iint_{(\mathbb{R}^d)^2} dy d\tilde{y} \, e^{-i \langle \beta, y \rangle}e^{i \langle \betati, \tilde{y} \rangle} \mathbb{E}\Big[\<Psi2>_n(t_1,y)\overline{\<Psi2>_n(t_2,\tilde{y})}\Big].\label{egality2}
\end{align}
Resorting to Wick's formula (see Lemma \ref{wick formula}), we can write
\begin{align}
\mathbb{E}\Big[\<Psi2>_n(t_1,y)\overline{\<Psi2>_n(t_2,\tilde{y})}\Big]= \Big|\mathbb{E}\Big[\<Psi>_n(t_1,y)\overline{\<Psi>_n(t_2,\tilde{y})}\Big]\Big|^2+ \Big|\mathbb{E}\Big[\<Psi>_n(t_1,y)\<Psi>_n(t_2,\tilde{y})\Big]\Big|^2.\label{egality3}
\end{align}
\noindent
Consequently, combining (\ref{egality2}) and (\ref{egality3}), and coming back to (\ref{egality1}), we obtain that
\begin{align}
&\mathbb{E}\bigg[\Big|\mathcal{F}^{-1}\Big(\big(1+|.|^2\big)^{-s+\frac{\kappa}{2}}\mathcal{F}\Big(\rho \Big[\<IPsi2>_{n}(t,.)-\<IPsi2>_{n}(s,.)\Big]\Big)\Big)(x)\Big|^2\bigg]\nonumber\\
&=\frac{1}{(2\pi)^{4d}} \iint_{(\mathbb{R}^d)^2}\frac{d\lambda d\lati}{(1+|\lambda|^2)^{s-\frac{\kappa}{2}}(1+|\lati|^2)^{s-\frac{\kappa}{2}}} e^{i \langle x, \lambda - \lati \rangle} \iint_{(\mathbb{R}^d)^2} d\beta d\betati \, \hat{\rho}(\lambda-\beta)\overline{\hat{\rho}(\lati-\betati)}\nonumber\\
&\bigg[\int_s^t dt_1 \int_0^t dt_2 \, e^{i(t-t_1)|\beta|^2}e^{-i(t-t_2)|\betati|^2}\iint_{(\mathbb{R}^d)^2} dy d\tilde{y} \, e^{-i \langle \beta, y \rangle}e^{i \langle \betati, \tilde{y} \rangle} \Big|\mathbb{E}\Big[\<Psi>_n(t_1,y)\overline{\<Psi>_n(t_2,\tilde{y})}\Big]\Big|^2\nonumber\\
&+\int_s^t dt_1 \int_0^t dt_2 \, e^{i(t-t_1)|\beta|^2}e^{-i(t-t_2)|\betati|^2}\iint_{(\mathbb{R}^d)^2} dy d\tilde{y} \, e^{-i \langle \beta, y \rangle}e^{i \langle \betati, \tilde{y} \rangle} \Big|\mathbb{E}\Big[\<Psi>_n(t_1,y)\<Psi>_n(t_2,\tilde{y})\Big]\Big|^2\nonumber\\
&+\int_0^s dt_1 \int_0^t dt_2 \, \Big(e^{i(t-t_1)|\beta|^2}-e^{i(s-t_1)|\beta|^2}\Big)e^{-i(t-t_2)|\betati|^2}\iint_{(\mathbb{R}^d)^2} dy d\tilde{y} \, e^{-i \langle \beta, y \rangle}e^{i \langle \betati, \tilde{y} \rangle} \Big|\mathbb{E}\Big[\<Psi>_n(t_1,y)\overline{\<Psi>_n(t_2,\tilde{y})}\Big]\Big|^2\nonumber\\
&+\int_0^s dt_1 \int_0^t dt_2 \, \Big(e^{i(t-t_1)|\beta|^2}-e^{i(s-t_1)|\beta|^2}\Big)e^{-i(t-t_2)|\betati|^2}\iint_{(\mathbb{R}^d)^2} dy d\tilde{y} \, e^{-i \langle \beta, y \rangle}e^{i \langle \betati, \tilde{y} \rangle} \Big|\mathbb{E}\Big[\<Psi>_n(t_1,y)\<Psi>_n(t_2,\tilde{y})\Big]\Big|^2\nonumber\\
&-\int_s^t dt_1 \int_0^s dt_2 \, e^{i(t-t_1)|\beta|^2}e^{-i(s-t_2)|\betati|^2}\iint_{(\mathbb{R}^d)^2} dy d\tilde{y} \, e^{-i \langle \beta, y \rangle}e^{i \langle \betati, \tilde{y} \rangle} \Big|\mathbb{E}\Big[\<Psi>_n(t_1,y)\overline{\<Psi>_n(t_2,\tilde{y})}\Big]\Big|^2\nonumber\\
&-\int_s^t dt_1 \int_0^s dt_2 \, e^{i(t-t_1)|\beta|^2}e^{-i(s-t_2)|\betati|^2}\iint_{(\mathbb{R}^d)^2} dy d\tilde{y} \, e^{-i \langle \beta, y \rangle}e^{i \langle \betati, \tilde{y} \rangle} \Big|\mathbb{E}\Big[\<Psi>_n(t_1,y)\<Psi>_n(t_2,\tilde{y})\Big]\Big|^2\nonumber\\
&-\int_0^s dt_1 \int_0^s dt_2 \, \Big(e^{i(t-t_1)|\beta|^2}-e^{i(s-t_1)|\beta|^2}\Big)e^{-i(s-t_2)|\betati|^2}\iint_{(\mathbb{R}^d)^2} dy d\tilde{y} \, e^{-i \langle \beta, y \rangle}e^{i \langle \betati, \tilde{y} \rangle} \Big|\mathbb{E}\Big[\<Psi>_n(t_1,y)\overline{\<Psi>_n(t_2,\tilde{y})}\Big]\Big|^2\nonumber\\
&-\int_0^s dt_1 \int_0^s dt_2 \, \Big(e^{i(t-t_1)|\beta|^2}-e^{i(s-t_1)|\beta|^2}\Big)e^{-i(s-t_2)|\betati|^2}\iint_{(\mathbb{R}^d)^2} dy d\tilde{y} \, e^{-i \langle \beta, y \rangle}e^{i \langle \betati, \tilde{y} \rangle} \Big|\mathbb{E}\Big[\<Psi>_n(t_1,y)\<Psi>_n(t_2,\tilde{y})\Big]\Big|^2\bigg].\label{huit}
\end{align}
The latter computation shows that $\mathbb{E}\bigg[\Big|\mathcal{F}^{-1}\Big(\big(1+|.|^2\big)^{-s+\frac{\kappa}{2}}\mathcal{F}\Big(\rho \Big[\<IPsi2>_{n}(t,.)-\<IPsi2>_{n}(s,.)\Big]\Big)\Big)(x)\Big|^2\bigg]$ is a sum of eight integrals. Let us focus our attention on the two following since the treatment of the six others is quite the same:
\begin{align}
&I_1=\frac{1}{(2\pi)^{4d}} \iint_{(\mathbb{R}^d)^2}\frac{d\lambda d\lati}{(1+|\lambda|^2)^{s-\frac{\kappa}{2}}(1+|\lati|^2)^{s-\frac{\kappa}{2}}} e^{i \langle x, \lambda - \lati \rangle} \iint_{(\mathbb{R}^d)^2} d\beta d\betati \, \hat{\rho}(\lambda-\beta)\overline{\hat{\rho}(\lati-\betati)}\nonumber\\
&\hspace{0,5cm}\int_0^s dt_1 \int_0^t dt_2 \, \Big(e^{i(t-t_1)|\beta|^2}-e^{i(s-t_1)|\beta|^2}\Big)e^{-i(t-t_2)|\betati|^2} \iint_{(\mathbb{R}^d)^2} dy d\tilde{y} \, e^{-i \langle \beta, y \rangle}e^{i \langle \betati, \tilde{y} \rangle} \Big|\mathbb{E}\Big[\<Psi>_n(t_1,y)\overline{\<Psi>_n(t_2,\tilde{y})}\Big]\Big|^2\nonumber
\end{align}
and
\begin{align}
&I_2=\frac{1}{(2\pi)^{4d}} \iint_{(\mathbb{R}^d)^2}\frac{d\lambda d\lati}{(1+|\lambda|^2)^{s-\frac{\kappa}{2}}(1+|\lati|^2)^{s-\frac{\kappa}{2}}} e^{i \langle x, \lambda - \lati \rangle} \iint_{(\mathbb{R}^d)^2} d\beta d\betati \, \hat{\rho}(\lambda-\beta)\overline{\hat{\rho}(\lati-\betati)}\nonumber \\
&\hspace{0,5cm}\int_0^s dt_1 \int_0^t dt_2 \, \Big(e^{i(t-t_1)|\beta|^2}-e^{i(s-t_1)|\beta|^2}\Big)e^{-i(t-t_2)|\betati|^2}\iint_{(\mathbb{R}^d)^2} dy d\tilde{y} \, e^{-i \langle \beta, y \rangle}e^{i \langle \betati, \tilde{y} \rangle} \Big|\mathbb{E}\Big[\<Psi>_n(t_1,y)\<Psi>_n(t_2,\tilde{y})\Big]\Big|^2.\nonumber
\end{align}
\newline
$\bullet$ \underline{Study of $I_1$:}
Using Definition \ref{luxn} of the covariance function of $\<Psi>_n$, we get that
\begin{align}
&\Big|\mathbb{E}\Big[\<Psi>_n(t_1,y)\overline{\<Psi>_n(t_2,\tilde{y})}\Big]\Big|^2=\bigg|\min(t_1,t_2)\frac{1}{(2\pi)^d}\int_{B_n}\frac{e^{i(t_1-t_2)|\xi|^2}}{(1+|\xi|^2)^{\alpha}}e^{-i\langle \xi, y - \tilde{y}\rangle}d\xi\bigg|^2\nonumber\\
&=\min(t_1,t_2)^2\frac{1}{(2\pi)^{2d}}\int_{B_n} \int_{B_n}d\xi_1 d\xi_2 \frac{e^{i(t_1-t_2)(|\xi_1|^2-|\xi_2|^2)}}{(1+|\xi_1|^2)^{\alpha}(1+|\xi_2|^2)^{\alpha}}e^{-i\langle \xi_1 -\xi_2, y - \tilde{y}\rangle}.\label{egality4}
\end{align}
Injecting $(\ref{egality4})$ into the definition of $I_1$, it holds that
\begin{align}
&I_1=\frac{1}{(2\pi)^{4d}} \iint_{(\mathbb{R}^d)^2}\frac{d\lambda d\lati}{(1+|\lambda|^2)^{s-\frac{\kappa}{2}}(1+|\lati|^2)^{s-\frac{\kappa}{2}}} e^{i \langle x, \lambda - \lati \rangle}\, \nonumber\\
&\hspace{0,5cm}\int_{B_n} \int_{B_n}\, \hat{\rho}(\lambda-(\xi_1-\xi_2))\overline{\hat{\rho}(\lati-(\xi_1-\xi_2))} \frac{d\xi_1 d\xi_2}{(1+|\xi_1|^2)^{\alpha}(1+|\xi_2|^2)^{\alpha}}\nonumber\\
&\hspace{0,5cm}\int_0^s dt_1 \int_0^t dt_2 \, \Big(e^{i(t-t_1)|\xi_1-\xi_2|^2}-e^{i(s-t_1)|\xi_1-\xi_2|^2}\Big)e^{-i(t-t_2)|\xi_1-\xi_2|^2}\min(t_1,t_2)^2 e^{i(t_1-t_2)(|\xi_1|^2-|\xi_2|^2)}.\nonumber
\end{align}
\newline
$\bullet$ \underline{Study of $I_2$:}
Using Definition \ref{luxn} of the covariance function of $\<Psi>_n$, we get that
\begin{align}
&\Big|\mathbb{E}\Big[\<Psi>_n(t_1,y)\<Psi>_n(t_2,\tilde{y})\Big]\Big|^2=\bigg|\frac{1}{(2\pi)^d}\int_{0}^{\min(t_1,t_2)}\int_{B_n}\frac{e^{i(t_1+t_2-2t')|\xi|^2}}{(1+|\xi|^2)^{\alpha}}e^{-i\langle \xi, y - \tilde{y}\rangle}d\xi dt'\bigg|^2\nonumber\\
&=\frac{1}{(2\pi)^{2d}}\int_{0}^{\min(t_1,t_2)}\int_{0}^{\min(t_1,t_2)}dt_1'dt_2'\int_{B_n} \int_{B_n}d\xi_1 d\xi_2 \frac{e^{i(t_1+t_2)(|\xi_1|^2-|\xi_2|^2)}e^{-2i t_1'|\xi_1|^2}e^{2i t_2'|\xi_2|^2}}{(1+|\xi_1|^2)^{\alpha}(1+|\xi_2|^2)^{\alpha}}e^{-i\langle \xi_1 -\xi_2, y - \tilde{y}\rangle}.\label{egality5}
\end{align}
Injecting $(\ref{egality5})$ into the definition of $I_2$, it holds that
\begin{align}
&I_2=\frac{1}{(2\pi)^{4d}} \iint_{(\mathbb{R}^d)^2}\frac{d\lambda d\lati}{(1+|\lambda|^2)^{s-\frac{\kappa}{2}}(1+|\lati|^2)^{s-\frac{\kappa}{2}}} e^{i \langle x, \lambda - \lati \rangle}  \, \nonumber \\
&\hspace{0,5cm}\int_{B_n}\int_{B_n}\, \hat{\rho}(\lambda-(\xi_1-\xi_2))\overline{\hat{\rho}(\lati-(\xi_1-\xi_2))}\frac{d\xi_1 d\xi_2}{(1+|\xi_1|^2)^{\alpha}(1+|\xi_2|^2)^{\alpha}}\nonumber\\
&\hspace{0,5cm}\int_0^s dt_1 \int_0^t dt_2 \, \Big(e^{i(t-t_1)|\xi_1-\xi_2|^2}-e^{i(s-t_1)|\xi_1-\xi_2|^2}\Big)e^{-i(t-t_2)|\xi_1-\xi_2|^2}e^{i(t_1+t_2)(|\xi_1|^2-|\xi_2|^2)}\nonumber\\
&\hspace{7cm}\int_{0}^{\min(t_1,t_2)}\int_{0}^{\min(t_1,t_2)}dt_1'dt_2' \, e^{-2i t_1'|\xi_1|^2}e^{2i t_2'|\xi_2|^2}.\nonumber
\end{align}
\newline
$\bullet$ \underline{Come back to (\ref{huit}):}
Combining the hypercontractivity of Wiener chaoses (see Lemma \ref{chaos}) and Lemma \ref{compact}, we get the following estimate
\begin{align*}
\int_{\mathbb{R}^d}dx\,\mathbb{E}\bigg[\Big|\mathcal{F}^{-1}\Big(\big(1+|.|^2\big)^{-s+\frac{\kappa}{2}}\mathcal{F}\Big(\rho \Big[\<IPsi2>_{n}(t,.)-\<IPsi2>_{n}(s,.)\Big]\Big)\Big)(x)\Big|^{2p}\bigg] 
\lesssim \Bigg(\sum_{k=1}^{8} |\tilde{I_k}|\Bigg)^p,
\end{align*}
where the eight integrals are the one related to the terms in (\ref{huit}). As above, let us focus on the treatment of the two first terms denoted by $\tilde{I_1}$ and $\tilde{I_2}$.\\
\newline
$\bullet$ \underline{Bound on $\tilde{I_1}$:}
Two changes of variables lead to the following equalities:
\begin{align}
&\tilde{I_1}=\int_{B_n} \int_{B_n}\, \frac{d\xi_1 d\xi_2}{(1+|\xi_1-\xi_2|^2)^{2s-\kappa}(1+|\xi_1|^2)^{\alpha}(1+|\xi_2|^2)^{\alpha}}\nonumber\\
&\hspace{1cm}\int_0^s dt_1 \int_0^t dt_2 \, \Big(e^{i(t-t_1)|\xi_1-\xi_2|^2}-e^{i(s-t_1)|\xi_1-\xi_2|^2}\Big)e^{-i(t-t_2)|\xi_1-\xi_2|^2}\min(t_1,t_2)^2 e^{i(t_1-t_2)(|\xi_1|^2-|\xi_2|^2)}\nonumber\\
&=\int_{B_n} \int_{B_n}\, \frac{d\xi_1 d\xi_2}{(1+|\xi_1+\xi_2|^2)^{2s-\kappa}(1+|\xi_1|^2)^{\alpha}(1+|\xi_2|^2)^{\alpha}}\nonumber\\
&\hspace{1cm}\int_0^s dt_1 \int_0^t dt_2 \, \Big(e^{i(t-t_1)|\xi_1+\xi_2|^2}-e^{i(s-t_1)|\xi_1+\xi_2|^2}\Big)e^{-i(t-t_2)|\xi_1+\xi_2|^2}\min(t_1,t_2)^2 e^{i(t_1-t_2)(|\xi_1|^2-|\xi_2|^2)}\nonumber\\
&=\int_{|\xi_1-\xi_2|\leq n} \int_{B_n}\, \frac{d\xi_1 d\xi_2}{(1+|\xi_1|^2)^{2s-\kappa}(1+|\xi_1-\xi_2|^2)^{\alpha}(1+|\xi_2|^2)^{\alpha}}\nonumber\\
&\hspace{1cm}\int_0^s dt_1 \int_0^t dt_2 \, \Big(e^{i(t-t_1)|\xi_1|^2}-e^{i(s-t_1)|\xi_1|^2}\Big)e^{-i(t-t_2)|\xi_1|^2}\min(t_1,t_2)^2 e^{i(t_1-t_2)(|\xi_1-\xi_2|^2-|\xi_2|^2)}.\nonumber
\end{align}
Let us introduce the additional notation: $\kappa(\underline{\xi})=|\xi_1|^2-|\xi_1-\xi_2|^2+|\xi_2|^2$.
Then, treating separately the cases where $t_1 \leq t_2$ and $t_2<t_1$, performing an integration by part in the second term of the following sum and using the mean value theorem, we obtain:
\begin{align}
&\bigg|\int_0^s dt_1 \int_0^t dt_2 \, \Big(e^{i(t-t_1)|\xi_1|^2}-e^{i(s-t_1)|\xi_1|^2}\Big)e^{-i(t-t_2)|\xi_1|^2}\min(t_1,t_2)^2 e^{i(t_1-t_2)(|\xi_1-\xi_2|^2-|\xi_2|^2)}\bigg|\nonumber\\
&\leq \bigg|\int_0^s dt_1 \int_{t_1}^t dt_2 \, \Big(e^{i(t-t_1)|\xi_1|^2}-e^{i(s-t_1)|\xi_1|^2}\Big)e^{-i(t-t_2)|\xi_1|^2}t_1^2 e^{i(t_1-t_2)(|\xi_1-\xi_2|^2-|\xi_2|^2)}\bigg|\nonumber\\
&\hspace{0,5cm}+\bigg|\int_0^s dt_1 \int_0^{t_1} dt_2 \, \Big(e^{i(t-t_1)|\xi_1|^2}-e^{i(s-t_1)|\xi_1|^2}\Big)e^{-i(t-t_2)|\xi_1|^2}t_2^2 e^{i(t_1-t_2)(|\xi_1-\xi_2|^2-|\xi_2|^2)}\bigg|\nonumber\\
&= \bigg|\int_0^s dt_1 \Big(e^{i(t-t_1)|\xi_1|^2}-e^{i(s-t_1)|\xi_1|^2}\Big)t_1^2 e^{-it|\xi_1|^2}e^{it_1(|\xi_1-\xi_2|^2-|\xi_2|^2)} \int_{t_1}^t dt_2 \, e^{it_2\kappa(\underline{\xi})}\bigg|\nonumber\\
&\hspace{0,5cm}+\bigg|\int_0^s dt_1 \Big(e^{i(t-t_1)|\xi_1|^2}-e^{i(s-t_1)|\xi_1|^2}\Big)e^{-it|\xi_1|^2}e^{it_1(|\xi_1-\xi_2|^2-|\xi_2|^2)}\int_0^{t_1} dt_2 \, t_2^2 e^{it_2\kappa(\underline{\xi})}\bigg|\nonumber\\
&\lesssim \frac{(t-s)^{\varepsilon}}{(1+|\xi_1|^2)^{-\varepsilon}(1+|\kappa(\underline{\xi})|)},\nonumber
\end{align}
for all $0< \varepsilon<1$.\\
Finally,
$$|\tilde{I_1}|\lesssim (t-s)^{\varepsilon}\iint_{(\mathbb{R}^d)^2}\, \frac{d\xi_1 d\xi_2}{(1+|\xi_1|^2)^{2s-\kappa-\varepsilon}(1+|\xi_1-\xi_2|^2)^{\alpha}(1+|\xi_2|^2)^{\alpha}(1+|\kappa(\underline{\xi})|)}.$$
\newline
$\bullet$ \underline{Bound on $\tilde{I_2}$:}
Two changes of variables lead to the following equalities:
\begin{align}
&\tilde{I_2}=\int_{B_n} \int_{B_n}\, \frac{d\xi_1 d\xi_2}{(1+|\xi_1-\xi_2|^2)^{2s-\kappa}(1+|\xi_1|^2)^{\alpha}(1+|\xi_2|^2)^{\alpha}}\nonumber\\
&\hspace{3cm}\int_0^s dt_1 \int_0^t dt_2 \, \Big(e^{i(t-t_1)|\xi_1-\xi_2|^2}-e^{i(s-t_1)|\xi_1-\xi_2|^2}\Big)e^{-i(t-t_2)|\xi_1-\xi_2|^2}e^{i(t_1+t_2)(|\xi_1|^2-|\xi_2|^2)}\nonumber\\
&\hspace{7cm}\int_{0}^{\min(t_1,t_2)}\int_{0}^{\min(t_1,t_2)}dt_1'dt_2' \, e^{-2i t_1'|\xi_1|^2}e^{2i t_2'|\xi_2|^2}\nonumber\\
&=\int_{B_n} \int_{B_n}\, \frac{d\xi_1 d\xi_2}{(1+|\xi_1+\xi_2|^2)^{2s-\kappa}(1+|\xi_1|^2)^{\alpha}(1+|\xi_2|^2)^{\alpha}}\nonumber\\
&\hspace{3cm}\int_0^s dt_1 \int_0^t dt_2 \, \Big(e^{i(t-t_1)|\xi_1+\xi_2|^2}-e^{i(s-t_1)|\xi_1+\xi_2|^2}\Big)e^{-i(t-t_2)|\xi_1+\xi_2|^2}e^{i(t_1+t_2)(|\xi_1|^2-|\xi_2|^2)}\nonumber\\
&\hspace{7cm}\int_{0}^{\min(t_1,t_2)}\int_{0}^{\min(t_1,t_2)}dt_1'dt_2' \, e^{-2i t_1'|\xi_1|^2}e^{2i t_2'|\xi_2|^2}\nonumber\\
&=\int_{B_n} \int_{B_n}\, \frac{d\xi_1 d\xi_2}{(1+|\xi_1|^2)^{2s-\kappa}(1+|\xi_1-\xi_2|^2)^{\alpha}(1+|\xi_2|^2)^{\alpha}}\nonumber\\
&\hspace{3cm}\int_0^s dt_1 \int_0^t dt_2 \, \Big(e^{i(t-t_1)|\xi_1|^2}-e^{i(s-t_1)|\xi_1|^2}\Big)e^{-i(t-t_2)|\xi_1|^2}e^{i(t_1+t_2)(|\xi_1-\xi_2|^2-|\xi_2|^2)}\nonumber\\
&\hspace{7cm}\int_{0}^{\min(t_1,t_2)}\int_{0}^{\min(t_1,t_2)}dt_1'dt_2' \, e^{-2i t_1'|\xi_1-\xi_2|^2}e^{2i t_2'|\xi_2|^2}.\nonumber
\end{align}
Let us introduce the additional notation: $\kappa_2(\underline{\xi})=|\xi_1|^2+|\xi_1-\xi_2|^2-|\xi_2|^2$.
Then, treating separately the cases where $t_1 \leq t_2$ and $t_2<t_1$, changing the order of integration and using the mean value theorem, we obtain:
\begin{align}
&\bigg|\int_0^s dt_1 \int_0^t dt_2 \, \Big(e^{i(t-t_1)|\xi_1|^2}-e^{i(s-t_1)|\xi_1|^2}\Big)e^{-i(t-t_2)|\xi_1|^2}e^{i(t_1+t_2)(|\xi_1-\xi_2|^2-|\xi_2|^2)}\nonumber\\
&\hspace{8cm}\int_{0}^{\min(t_1,t_2)}\int_{0}^{\min(t_1,t_2)}dt_1'dt_2' \, e^{-2i t_1'|\xi_1-\xi_2|^2}e^{2i t_2'|\xi_2|^2}\bigg|\nonumber\\
&\leq \bigg|\int_0^s dt_1 \int_{t_1}^t dt_2 \, \Big(e^{i(t-t_1)|\xi_1|^2}-e^{i(s-t_1)|\xi_1|^2}\Big)e^{-i(t-t_2)|\xi_1|^2}e^{i(t_1+t_2)(|\xi_1-\xi_2|^2-|\xi_2|^2)}\nonumber\\
&\hspace{8cm}\int_{0}^{t_1}\int_{0}^{t_1}dt_1'dt_2' \, e^{-2i t_1'|\xi_1-\xi_2|^2}e^{2i t_2'|\xi_2|^2}\bigg|\nonumber\\
&+\bigg|\int_0^s dt_1 \int_0^{t_1} dt_2 \, \Big(e^{i(t-t_1)|\xi_1|^2}-e^{i(s-t_1)|\xi_1|^2}\Big)e^{-i(t-t_2)|\xi_1|^2}e^{i(t_1+t_2)(|\xi_1-\xi_2|^2-|\xi_2|^2)}\nonumber\\
&\hspace{8cm}\int_{0}^{t_2}\int_{0}^{t_2}dt_1'dt_2' \, e^{-2i t_1'|\xi_1-\xi_2|^2}e^{2i t_2'|\xi_2|^2}\bigg|\nonumber\\
&=\bigg|\int_0^s dt_1  \, \Big(e^{i(t-t_1)|\xi_1|^2}-e^{i(s-t_1)|\xi_1|^2}\Big)e^{-it|\xi_1|^2}e^{it_1(|\xi_1-\xi_2|^2-|\xi_2|^2)}\nonumber\\
&\hspace{8cm}\int_{0}^{t_1}\int_{0}^{t_1}dt_1'dt_2' \, e^{-2i t_1'|\xi_1-\xi_2|^2}e^{2i t_2'|\xi_2|^2}\int_{t_1}^t dt_2 \, e^{it_2\kappa_2(\underline{\xi})}\bigg|\nonumber\\
&\hspace{0,5cm}+\bigg|\int_0^s dt_1 \, \Big(e^{i(t-t_1)|\xi_1|^2}-e^{i(s-t_1)|\xi_1|^2}\Big)e^{-it|\xi_1|^2}e^{it_1(|\xi_1-\xi_2|^2-|\xi_2|^2)}\nonumber\\
&\hspace{8cm}\int_{0}^{t_1}\int_{0}^{t_1}dt_1'dt_2' \, e^{-2i t_1'|\xi_1-\xi_2|^2}e^{2i t_2'|\xi_2|^2}\int_{\max(t_1',t_2')}^{t_1} dt_2 \, e^{it_2\kappa_2(\underline{\xi})}\bigg|\nonumber\\
&\lesssim \frac{(t-s)^{\varepsilon}}{(1+|\xi_1|^2)^{-\varepsilon}(1+|\kappa_2(\underline{\xi})|)},\nonumber
\end{align}
for all $0< \varepsilon<1$.\\
Finally,
\begin{align}
&|\tilde{I_2}|\lesssim (t-s)^{\varepsilon}\iint_{(\mathbb{R}^d)^2}\, \frac{d\xi_1 d\xi_2}{(1+|\xi_1|^2)^{2s-\kappa-\varepsilon}(1+|\xi_1-\xi_2|^2)^{\alpha}(1+|\xi_2|^2)^{\alpha}(1+|\kappa_2(\underline{\xi})|)}\nonumber\\
&\hspace{0,5cm}=(t-s)^{\varepsilon}\iint_{(\mathbb{R}^d)^2}\, \frac{d\xi_1 d\xi_2}{(1+|\xi_1|^2)^{2s-\kappa-\varepsilon}(1+|\xi_1-\xi_2|^2)^{\alpha}(1+|\xi_2|^2)^{\alpha}(1+|\kappa(\underline{\xi})|)}.\nonumber
\end{align}
The end of the proof consists in showing that the integral 
$$\mathbbm{I}=\iint_{(\mathbb{R}^d)^2} \frac{d\xi_1 d\xi_2}{(1+|\xi_1|^2)^{2s-\kappa}(1+|\xi_1-\xi_2|^2)^{\alpha}(1+|\xi_2|^2)^{\alpha}(1+|\kappa(\underline{\xi})|)}$$
is finite.\\
We split $\mathbb{R}^d \times \mathbb{R}^d$ into two domains defined by
$$D_1=\{(\xi_1,\xi_2)\in \mathbb{R}^d\times \mathbb{R}^d,|\xi_1|\leq 1\}\quad \text{and} \quad D_2=\{(\xi_1,\xi_2)\in \mathbb{R}^d\times\mathbb{R}^d,|\xi_1|> 1\}$$
and we denote by $\mathbbm{I}_1$ and $\mathbbm{I}_2$ the associated integrals.\\
\newline
$\bullet $ \underline{Bound on $\mathbbm{I}_1$:}\\
It holds that:
\begin{align}
&\mathbbm{I}_1=\int_{|\xi_1|\leq 1} \int_{\mathbb{R}^d}\frac{d\xi_1 d\xi_2}{(1+|\xi_1|^2)^{2s-\kappa}(1+|\xi_1-\xi_2|^2)^{\alpha}(1+|\xi_2|^2)^{\alpha}(1+|\kappa(\underline{\xi})|)}\nonumber\\
&\hspace{0,7cm}\lesssim \int_{|\xi_1|\leq 1} \int_{\mathbb{R}^d}\frac{d\xi_1 d\xi_2}{(1+|\xi_1|^2)^{2s-\kappa}(1+|\xi_1-\xi_2|^2)^{\alpha}(1+|\xi_2|^2)^{\alpha}}\nonumber\\
&\hspace{0,7cm}\lesssim \int_{|\xi_1|\leq 1} \frac{d\xi_1}{(1+|\xi_1|^2)^{2s-\kappa}(1+|\xi_1|^2)^{2\alpha-\frac{d}{2}}}\nonumber\\
&\hspace{0,7cm}\lesssim 1 \label{I1}
\end{align}
where we have used Lemma \ref{convolution} with the condition $\frac{d}{4}<\alpha<\frac{d}{2}$ to derive the second inequality.\\
\newline
$\bullet $ \underline{Bound on $\mathbbm{I}_2$:}\\
We split $D_2$ into two domains defined by
$$D_2^{1}=\{(\xi_1,\xi_2)\in D_2,|\xi_1-\xi_2|< |\xi_2|\}\quad \text{and} \quad D_2^{2}=\{(\xi_1,\xi_2)\in D_2,|\xi_1-\xi_2|\geq |\xi_2|\},$$
and we denote by $\mathbbm{I}^{1}_{2}$ and $\mathbbm{I}^{2}_{2}$ the associated integrals.\\
\newline
$\bullet $ \underline{Bound on $\mathbbm{I}^{1}_2$:}\\
Let us observe that on $D_2^{1}$, the following bound holds true: $|\kappa(\underline{\xi})|\geq |\xi_1|^2$. Now, resorting to Lemma \ref{convolution} again, one has
\begin{align}
&\mathbbm{I}_2^{1}=\int_{|\xi_1|> 1} \int_{\mathbb{R}^d}\frac{\mathbbm{1}_{\{|\xi_1-\xi_2|< |\xi_2|\}}d\xi_1 d\xi_2}{(1+|\xi_1|^2)^{2s-\kappa}(1+|\xi_1-\xi_2|^2)^{\alpha}(1+|\xi_2|^2)^{\alpha}(1+|\kappa(\underline{\xi})|)}\nonumber\\
&\hspace{0,7cm}\lesssim \int_{|\xi_1|> 1} \int_{\mathbb{R}^d}\frac{d\xi_1 d\xi_2}{(1+|\xi_1|^2)^{2s-\kappa+1}(1+|\xi_1-\xi_2|^2)^{\alpha}(1+|\xi_2|^2)^{\alpha}}\nonumber\\
&\hspace{0,7cm}\lesssim \int_{|\xi_1|>1} \frac{d\xi_1}{(1+|\xi_1|^2)^{2s-\kappa+1+2\alpha-\frac{d}{2}}}\nonumber\\
&\hspace{0,7cm}\lesssim \int_{1}^{+\infty} \frac{dr}{r^{4s-2\kappa+2+4\alpha-2d+1}}\nonumber\\
&\hspace{0,7cm}\lesssim 1 \label{I21}
\end{align}
where we have used the fact that $4s-2\kappa+2+4\alpha-2d+1>1$ resulting from $s>\frac{d}{2}-\alpha$ to obtain the last inequality.\\
\newline
$\bullet $ \underline{Bound on $\mathbbm{I}^{2}_2$:}\\
Using the law of cosines, we write
$$\kappa(\underline{\xi})=|\xi_1|^2-|\xi_1-\xi_2|^2+|\xi_2|^2=2|\xi_1||\xi_2|\cos(\theta),$$
where $\theta=\angle(\xi_1,\xi_2) \in [0,2\pi[$.\\
We split $D_2^{2}$ into two domains defined by
$$D_2^{2,1}=\{(\xi_1,\xi_2)\in D_2^{2},|\cos(\theta)|\gtrsim 1\}\quad \text{and} \quad D_2^{2,2}=\{(\xi_1,\xi_2)\in D_2^{2},|\cos(\theta)|\ll 1\}$$
and we denote by $\mathbbm{I}^{2,1}_2$ and $\mathbbm{I}^{2,2}_2$ the associated integrals.\\
\newline
$\bullet $ \underline{Bound on $\mathbbm{I}^{2,1}_2$:}\\
The aim is to prove that the integral
$$\mathbbm{I}^{2,1}_2=\int_{|\xi_1|> 1} \int_{\mathbb{R}^d}\frac{\mathbbm{1}_{\{|\xi_1-\xi_2|\geq |\xi_2|\}}\mathbbm{1}_{\{|\cos(\theta)| \gtrsim 1\}}d\xi_1 d\xi_2}{(1+|\xi_1|^2)^{2s-\kappa}(1+|\xi_1-\xi_2|^2)^{\alpha}(1+|\xi_2|^2)^{\alpha}(1+|\kappa(\underline{\xi})|)}$$
is finite. 
By dyadically decomposing $\xi_2$ into $|\xi_2|\sim N_2$ for dyadic numbers $N_2$ (ie $N_2=2^j$ for $j\geq 1$ or $N_2=2^{-j}$ for $j\geq 1$), it holds that
\begin{align}
&\mathbbm{I}^{2,1}_2\lesssim \int_{|\xi_1|> 1}\sum_{\underset{\text{dyadic}}{N_2}}\int_{\mathbb{R}^d}\frac{\mathbbm{1}_{\{|\xi_2|\sim N_2\}}\mathbbm{1}_{\{|\xi_1-\xi_2|\geq |\xi_2|\}}d\xi_1d\xi_2}{(1+|\xi_1|^2)^{2s-\kappa}(1+|\xi_1-\xi_2|^2)^{\alpha}(1+|\xi_2|^2)^{\alpha}(1+|\xi_1||\xi_2|)}\nonumber\\
&\lesssim \int_{|\xi_1|> 1}\sum_{\underset{\text{dyadic}}{N_2}}\int_{\mathbb{R}^d}\frac{\mathbbm{1}_{\{|\xi_2|\sim N_2\}}d\xi_1d\xi_2}{(1+|\xi_1|^2)^{2s-\kappa}\max(|\xi_1|,N_2)^{2\alpha}(1+N_2^2)^{\alpha}(1+|\xi_1|N_2)},\nonumber\\
&\lesssim \int_{|\xi_1|> 1}d\xi_1\sum_{\underset{\text{dyadic}}{N_2}}\frac{N_2^d}{(1+|\xi_1|^2)^{2s-\kappa}\max(|\xi_1|,N_2)^{2\alpha}(1+N_2^2)^{\alpha}(1+|\xi_1|N_2)},
\label{I221c}
\end{align}
where we haved used the fact that $|\xi_1-\xi_2|\sim \max(|\xi_1|,|\xi_2|)$ to derive the second inequality.\\
Or, for $0<\varepsilon<1$,
\begin{align}
&\sum_{\underset{\text{dyadic}}{N_2}}\frac{N_2^d}{(1+|\xi_1|^2)^{2s-\kappa}\max(|\xi_1|,N_2)^{2\alpha}(1+N_2^2)^{\alpha}(1+|\xi_1|N_2)}\nonumber\\
&\lesssim\sum_{\underset{\text{dyadic}}{N_2}}\frac{N_2^d }{(1+|\xi_1|^2)^{2s-\kappa}\max(|\xi_1|,N_2)^{2\alpha}(1+N_2^2)^{\alpha}|\xi_1|^{1-\varepsilon}N_2^{1-\varepsilon}}\nonumber\\
&\lesssim\sum_{\underset{\text{dyadic}}{N_2}}\frac{N_2^{d-1+\varepsilon} }{(1+|\xi_1|^2)^{2s-\kappa}\max(|\xi_1|,N_2)^{2\alpha}(1+N_2^2)^{\alpha}|\xi_1|^{1-\varepsilon}}\nonumber\\
&\lesssim\sum_{\underset{\text{dyadic}}{N_2 \geq 1}}\frac{1 }{(1+|\xi_1|^2)^{2s-\kappa+\frac{1}{2}-\frac{\varepsilon}{2}}\max(|\xi_1|,N_2)^{2\alpha}N_2^{2\alpha-d+1-\varepsilon}}\nonumber\\
&\hspace{0,7cm}+\sum_{\underset{\text{dyadic}}{N_2 \leq 1}}\frac{N_2^{d-1+\varepsilon} }{(1+|\xi_1|^2)^{2s-\kappa+\alpha+\frac{1}{2}-\frac{\varepsilon}{2}}}\nonumber\\
&\lesssim \sum_{\underset{\text{dyadic}}{1 \leq N_2 < |\xi_1|}} \frac{1}{(1+|\xi_1|^2)^{2s-\kappa+\alpha+\frac{1}{2}-\frac{\varepsilon}{2}}N_2^{2\alpha-d+1-\varepsilon}} + \sum_{\underset{\text{dyadic}}{N_2 \geq |\xi_1|}}\frac{1}{(1+|\xi_1|^2)^{2s-\kappa+\frac{1}{2}-\frac{\varepsilon}{2}}N_2^{4\alpha-d+1-\varepsilon}}\label{sum1}\\
&\hspace{0,7cm}+\frac{1}{(1+|\xi_1|^2)^{2s-\kappa+\alpha+\frac{1}{2}-\frac{\varepsilon}{2}}}.\nonumber
\end{align}
Let us focus our attention on the first term in (\ref{sum1}). When $d=1$ or $d=2$, as $\alpha > \frac{d-1}{2}$, $2\alpha-d+1-\varepsilon \geq 0$ for $\varepsilon>0$ small enough and this term is bounded by
$$ \sum_{\underset{\text{dyadic}}{1 \leq N_2 < |\xi_1|}} \frac{1}{(1+|\xi_1|^2)^{2s-\kappa+\alpha+\frac{1}{2}-\frac{\varepsilon}{2}}N_2^{2\alpha-d+1-\varepsilon}}\lesssim \frac{\ln(|\xi_1|)}{(1+|\xi_1|^2)^{2s-\kappa+\alpha+\frac{1}{2}-\frac{\varepsilon}{2}}}.$$
Now, let us remark that, when $d=3$, $\varepsilon$ can be removed from the previous computations. When $\alpha \geq 1$, the following bound holds true
$$ \sum_{\underset{\text{dyadic}}{1 \leq N_2 < |\xi_1|}} \frac{1}{(1+|\xi_1|^2)^{2s-\kappa+\alpha+\frac{1}{2}}N_2^{2\alpha-d+1}}\lesssim \frac{\ln(|\xi_1|)}{(1+|\xi_1|^2)^{2s-\kappa+\alpha+\frac{1}{2}}},$$
whereas, when $\alpha < 1$, one has
$$ \sum_{\underset{\text{dyadic}}{1 \leq N_2 < |\xi_1|}} \frac{1}{(1+|\xi_1|^2)^{2s-\kappa+\alpha+\frac{1}{2}}N_2^{2\alpha-d+1}}\lesssim \frac{1}{(1+|\xi_1|^2)^{2s-\kappa+2\alpha+1-\frac{d}{2}}}.$$
Now, for all $1\leq d \leq 3$, as $\alpha > \frac{d-1}{4}$, $4\alpha-d+1-\varepsilon>0$ for $\varepsilon>0$ small enough and the second term in (\ref{sum1}) is bounded by
$$\sum_{\underset{\text{dyadic}}{N_2 \geq |\xi_1|}}\frac{1}{(1+|\xi_1|^2)^{2s-\kappa+\frac{1}{2}-\frac{\varepsilon}{2}}N_2^{4\alpha-d+1-\varepsilon}}\lesssim \frac{1}{(1+|\xi_1|^2)^{2s-\kappa+2\alpha+1-\frac{d}{2}-\varepsilon}}.$$
Finally, coming back to (\ref{I221c}), when $d=1$ or $d=2$,
\begin{align*}
\mathbbm{I}_{2}^{2,1}\lesssim \int_{|\xi_1|> 1} \frac{\ln(|\xi_1|)d\xi_1}{(1+|\xi_1|^2)^{2s-\kappa+\alpha+\frac{1}{2}-\frac{\varepsilon}{2}}}+ \int_{|\xi_1|> 1} \frac{d\xi_1}{(1+|\xi_1|^2)^{2s-\kappa+2\alpha+1-\frac{d}{2}-\varepsilon}} + \int_{|\xi_1|> 1} \frac{d\xi_1}{(1+|\xi_1|^2)^{2s-\kappa+\alpha+\frac{1}{2}-\frac{\varepsilon}{2}}}.
\end{align*}
When $d=3$ and $\alpha \geq 1$,
\begin{align*}
\mathbbm{I}_{2}^{2,1}\lesssim \int_{|\xi_1|> 1} \frac{\ln(|\xi_1|)d\xi_1}{(1+|\xi_1|^2)^{2s-\kappa+\alpha+\frac{1}{2}}}+ \int_{|\xi_1|> 1} \frac{d\xi_1}{(1+|\xi_1|^2)^{2s-\kappa+2\alpha+1-\frac{d}{2}}} + \int_{|\xi_1|> 1} \frac{d\xi_1}{(1+|\xi_1|^2)^{2s-\kappa+\alpha+\frac{1}{2}}},
\end{align*}
and, when $d=3$ and $\alpha < 1$,
\begin{align*}
\mathbbm{I}_{2}^{2,1}\lesssim \int_{|\xi_1|> 1} \frac{d\xi_1}{(1+|\xi_1|^2)^{2s-\kappa+2\alpha+1-\frac{d}{2}}} + \int_{|\xi_1|> 1} \frac{d\xi_1}{(1+|\xi_1|^2)^{2s-\kappa+\alpha+\frac{1}{2}}}.
\end{align*}
In all cases,  according to the value of $\kappa$ and since $s>\frac{d}{2}-\alpha$,
\begin{align}
\mathbbm{I}_{2}^{2,1}\lesssim 1 \label{I221}
\end{align}
as soon as  $\varepsilon>0$ is small enough.\\
\newline
$\bullet $ \underline{Bound on $\mathbbm{I}^{2,2}_2$:}\\
First of all, we can remark that, when $d=1$, $\mathbbm{I}^{2,2}_2=0$ since, for all $(\xi_1,\xi_2)\in \mathbb{R}^2$, $\cos(\angle(\xi_1,\xi_2))=~1$ or $-1$. Consequently, in the following, $d=2$ or $3$. As $|\cos(\theta)| \ll 1$, $\displaystyle\bigg|\frac{\pi}{2}-\theta\bigg| \ll1$ or $\displaystyle\bigg|\frac{3\pi}{2}-\theta\bigg| \ll1$. We will only deal with the case where $\displaystyle\bigg|\frac{\pi}{2}-\theta\bigg| \ll1$ since the other case is quite the same (it suffices to replace $\displaystyle\bigg|\frac{\pi}{2}-\theta\bigg|$ by $\displaystyle\bigg|\frac{3\pi}{2}-\theta\bigg|$ in all the computations). By dyadically decomposing $\displaystyle\bigg|\frac{\pi}{2}-\theta\bigg|\sim 2^{-k}$ for $k \geq 0$ and $\xi_2$ into $|\xi_2|\sim N_2$ for dyadic numbers $N_2$ (ie $N_2=2^j$ for $j\geq 1$ or $N_2=2^{-j}$ for $j\geq 1$) again, we get that
\begin{align}
&\mathbbm{I}^{2,2}_2\lesssim \int_{|\xi_1|> 1}\sum_{\underset{\text{dyadic}}{N_2}}\sum_{k=0}^{+\infty}\int_{\mathbb{R}^d}\frac{\mathbbm{1}_{\{|\xi_2|\sim N_2\}}\mathbbm{1}_{\{|\frac{\pi}{2}-\theta|\ll 1\}}\mathbbm{1}_{\{|\frac{\pi}{2}-\theta| \sim 2^{-k}\}}\mathbbm{1}_{\{|\xi_1-\xi_2|\geq |\xi_2|\}}d\xi_1d\xi_2}{(1+|\xi_1|^2)^{2s-\kappa}(1+|\xi_1-\xi_2|^2)^{\alpha}(1+|\xi_2|^2)^{\alpha}(1+|\kappa(\underline{\xi})|)}\nonumber\\
&\lesssim \int_{|\xi_1|> 1}\sum_{\underset{\text{dyadic}}{N_2}}\sum_{k=0}^{+\infty}\int_{\mathbb{R}^d}\frac{\mathbbm{1}_{\{|\xi_2|\sim N_2\}}\mathbbm{1}_{\{|\frac{\pi}{2}-\theta|\ll 1\}}\mathbbm{1}_{\{|\frac{\pi}{2}-\theta| \sim 2^{-k}\}}d\xi_1d\xi_2}{(1+|\xi_1|^2)^{2s-\kappa}\max(|\xi_1|,N_2)^{2\alpha}(1+N_2^2)^{\alpha}(1+|\xi_1|N_22^{-k})},\label{I222c}
\end{align}
where we haved used the fact that $|\xi_1-\xi_2|\sim \max(|\xi_1|,|\xi_2|)$ to derive the last inequality.\\
Suppose that $d=2$. Then, for every $\xi_1 \in \mathbb{R}^2$, we see that the set of possible $\xi_2$ with $|\xi_2|\sim N_2$, $\displaystyle\bigg|\frac{\pi}{2}-\theta\bigg|\ll 1$ and $\displaystyle\bigg|\frac{\pi}{2}-\theta\bigg| \sim 2^{-k}$ is included in an axially symmetric trapezoid $\mathcal{D}$ whose height is $\sim N_2$ and the top and bottom widths are $\sim2^{-k} N_2$ with an axis of symmetry given by $(\mathbb{R}\xi_1)^{\perp}$. Consequently, $\text{vol}(\mathcal{D})\sim N_2^2 2^{-k}$ and
$$\int_{\mathbb{R}^2}\mathbbm{1}_{\{|\xi_2|\sim N_2\}}\mathbbm{1}_{\{|\frac{\pi}{2}-\theta|\ll 1\}}\mathbbm{1}_{\{|\frac{\pi}{2}-\theta| \sim 2^{-k}\}} d\xi_2 \lesssim N_2^2 2^{-k}.$$
Suppose now that $d=3$. Then, this time, for every $\xi_1 \in \mathbb{R}^3$, the set of possible $\xi_2$ with $|\xi_2|\sim N_2$, $\displaystyle\bigg|\frac{\pi}{2}-\theta\bigg|\ll 1$ and $\displaystyle\bigg|\frac{\pi}{2}-\theta\bigg| \sim 2^{-k}$ is included in a cone $\mathcal{D}$ whose height is $\sim N_2$, whose radius  of base disc is $\sim2^{-k} N_2$ and that presents a symmetry given by the plan $(\mathbb{R}\xi_1)^{\perp}$. Thus, $\text{vol}(\mathcal{D})\sim N_2^3 2^{-2k}$ and
$$\int_{\mathbb{R}^3}\mathbbm{1}_{\{|\xi_2|\sim N_2\}}\mathbbm{1}_{\{|\frac{\pi}{2}-\theta|\ll 1\}}\mathbbm{1}_{\{|\frac{\pi}{2}-\theta| \sim 2^{-k}\}} d\xi_2 \lesssim N_2^3 2^{-2k}.$$
To sum up,
$$\int_{\mathbb{R}^d}\mathbbm{1}_{\{|\xi_2|\sim N_2\}}\mathbbm{1}_{\{|\frac{\pi}{2}-\theta|\ll 1\}}\mathbbm{1}_{\{|\frac{\pi}{2}-\theta| \sim 2^{-k}\}} d\xi_2 \lesssim N_2^d 2^{-(d-1)k}.$$
Coming back to (\ref{I222c}), we deduce that, for $0<\varepsilon<1$,
\begin{align}
&\sum_{\underset{\text{dyadic}}{N_2}}\sum_{k=0}^{+\infty}\int_{\mathbb{R}^d}\frac{\mathbbm{1}_{\{|\xi_2|\sim N_2\}}\mathbbm{1}_{\{|\frac{\pi}{2}-\theta| \sim 2^{-k}\}}d\xi_2}{(1+|\xi_1|^2)^{2s-\kappa}\max(|\xi_1|,N_2)^{2\alpha}(1+N_2^2)^{\alpha}(1+|\xi_1|N_22^{-k})}\nonumber\\
&\lesssim\sum_{\underset{\text{dyadic}}{N_2}}\sum_{k=0}^{+\infty}\frac{N_2^d 2^{-(d-1)k}}{(1+|\xi_1|^2)^{2s-\kappa}\max(|\xi_1|,N_2)^{2\alpha}(1+N_2^2)^{\alpha}|\xi_1|^{1-\varepsilon}N_2^{1-\varepsilon}2^{-k+\varepsilon k}}\nonumber\\
&\lesssim\sum_{\underset{\text{dyadic}}{N_2}}\frac{N_2^{d-1+\varepsilon} }{(1+|\xi_1|^2)^{2s-\kappa}\max(|\xi_1|,N_2)^{2\alpha}(1+N_2^2)^{\alpha}|\xi_1|^{1-\varepsilon}}.\nonumber\\
\end{align}
Now, we are dealing with the same sum as in the computations related to $\mathbbm{I}^{2,1}_2$. We can mimic the arguments to obtain that
\begin{align}
\mathbbm{I}^{2,2}_2 \lesssim 1.\label{I222}
\end{align}
\newline
$\bullet $ \underline{Come back to $\mathbbm{I}$:}\\
Combining the four bounds (\ref{I1}), (\ref{I21}), (\ref{I221}) and (\ref{I222}), we deduce that the integral $\mathbb{I}$ is finite. That concludes the proof.
\end{proof}

\section{Deterministic analysis of equation (\ref{dim-d-sch})}\label{point fixe}

Let $1 \leq d \leq 3$ be a space dimension and $T>0$ a positive time. Fix $\alpha$ a real number verifying $$\displaystyle\frac{d}{4}<\alpha<\frac{d}{2}.$$ 
The aim of this section is to establish the local well-posedness of equation (\ref{dim-d-sch}), that is to present the proof of Theorem~\ref{major}. Let us remember that our model is understood in the sense of Definition~\ref{sol-sch2}. In other words, we have to prove that there exists $v$ verifying the fixed-point equality
\begin{multline}\label{fixed-point}
v_t=e^{-it\Delta}\phi-i\int_0^t e^{-i(t-\tau)\Delta}(\rho^2 |v_{\tau}|^2)\, d\tau-i\int_0^t e^{-i(t-\tau)\Delta}((\rho \overline{v}_{\tau})\cdot(\rho \<Psi>_{\tau}))\, d\tau\\
-i\int_0^t e^{-i(t-\tau)\Delta}((\rho v_{\tau})\cdot(\overline{\rho \<Psi>_{\tau}}))\, d\tau+\rho^2\<IPsi2> , \quad t\in [0,T] \, ,
\end{multline}
where the processes $\rho \<Psi>$ and $\rho^2 \<IPsi2>$ are defined through Proposition~\ref{luxo1} and Proposition~\ref{cerise-ssj}. Recall that, for all $p \geq 2$,
$$\rho \<Psi> \in \mathcal{C}([0,T];\mathcal{W}^{-(\frac{d}{2}-\alpha)-\varepsilon,p}(\mathbb{R}^d)),$$
and that, thanks to multilinear smoothing,
$$\rho^2 \<IPsi2> \in  \mathcal{C}([0,T];\mathcal{W}^{-(d-2\alpha)-\varepsilon+ \kappa,p}(\mathbb{R}^d)),$$
for every $\varepsilon>0$ and where
\begin{equation*}
\kappa=
\left\{
\begin{array}{l}
1-\alpha \quad \text{if} \ d=1\\
\displaystyle\frac{3}{2}-\alpha \quad \text{if} \ d=2\\
2-\alpha \quad \text{if} \ d=3 \ \text{and} \ \alpha\geq 1 \\
1 \quad \text{if} \ d=3 \ \text{and} \ \alpha< 1 \  .
\end{array}
\right.
\end{equation*}
Fix $\varepsilon>0$ and let $s>0$ the real number defined by $s=s_{d,\alpha,\varepsilon}=(\frac{d}{2}-\alpha)+\varepsilon$.
In the following, we will use a deterministic approach to deal with equation (\ref{fixed-point}) and we will henceforth consider the pair $(\rho \<Psi>,\rho^2 \<IPsi2>)$ as a given element in the space
\begin{equation*}
\mathcal{E}_{s,\kappa}:=\bigcap_{2\leq p< \infty}\mathcal{C}([0,T];\mathcal{W}^{-s,p}(\mathbb{R}^d)) \times \mathcal{C}([0,T];\mathcal{W}^{-2s+\kappa,p}(\mathbb{R}^d))\, ,
\end{equation*}
and then try to solve the more general deterministic equation: for $(\luxor,\cherry) \in \mathcal{E}_{s,\kappa}$,
\begin{multline}\label{ccl-snls}
v_t=e^{-it\Delta}\phi-i\int_0^t e^{-i(t-\tau)\Delta}(\rho^2 |v_{\tau}|^2)\, d\tau-i\int_0^t e^{-i(t-\tau)\Delta}((\rho \overline{v}_{\tau})\cdot(\luxor_{\tau}))\, d\tau\\
-i\int_0^t e^{-i(t-\tau)\Delta}((\rho v_{\tau})\cdot(\overline{\luxor_{\tau}}))\, d\tau+\cherry , \quad t\in [0,T] \, .
\end{multline}
\\ 
In the next subsection, we present the different tools that will permit us to deal with each term of the above equation.

\subsection{Technical lemmas}

\subsubsection{About the Schr\"{o}dinger operator}
Before stating the well-known Strichartz inequalities, let us define the notion of Schr\"{o}dinger admissible pair.
\begin{definition}
A pair $(p,q) \in [2,+\infty]^2$ is said to be Schr\"{o}dinger admissible if $$ (p,q,d)\neq (2,+\infty,2)\hspace{0,2cm}\text{and}\hspace{0,2cm} \frac{2}{p}+\frac{d}{q}=\frac{d}{2}.$$ 
\end{definition}

\begin{lemma}[Strichartz inequalities, see~{\cite[Paragraph 2.3]{caz-ter}}]\label{lem:strichartz2}
Fix $d\geq 1$ a space dimension and $s \in \mathbb{R}$. Let $u$ stand for the mild solution of equation
\begin{equation*}
\left\{
\begin{array}{l}
i\partial_t u(t,x)-\Delta u(t,x)= F(t,x) \, , \quad  t\in [0,T] \, , \, x\in \R^d \, ,\\
u(0,x)=\phi(x)\, .
\end{array}
\right.
\end{equation*}
Then, for all Schr\"{o}dinger admissible pairs $(p,q)$ and $(a,b)$,  it holds that
\begin{equation*}
\|u\|_{L^p([0,T]; \mathcal{W}^{s,q}(\mathbb{R}^{d}))}\lesssim \|\phi\|_{H^s(\mathbb{R}^d)}+\|F\|_{L^{a'}([0,T];\mathcal{W}^{s,b'}(\mathbb{R}^d))}\, ,
\end{equation*}
where the notations $a',b'$ refer to the H{\"o}lder conjugates of $a,b$.
\end{lemma}

\begin{remark}
Strichartz inequality is a fundamental tool to deal with nonlinear Schr\"{o}dinger equations. It presents the advantage to provide a gain of integrability but the drawback that the solution inherit the (potentially bad) regularity of the initial condition and of the second member $F$. 
\end{remark}
The Schr\"{o}dinger operator can generate a local gain of regularity. This latter is described in \cite{Schaeffer1} where the authors have generalized the so-called Kato smoothing effect that offers locally a gain of a half of a derivative (understood in the Sobolev meaning).

\begin{lemma}\label{local-regu}
Fix $d\geq 1$ a space dimension. Let $\rho:\R^d \to \R$ be of the form~\eqref{rho}, $0\leq s,\eta \leq\frac{1}{2}$ and $0\leq T\leq 1$. Suppose that $\phi \in H^{-s}(\mathbb{R}^d)$ and $F \in L^1([0,T];H^{-s}(\mathbb{R}^d))$. Let $u$ stand for the mild solution of equation 
\begin{equation*}
\left\{
\begin{array}{l}
i\partial_t u(t,x)-\Delta u(t,x)= F(t,x) \, , \quad  t\in [0,T] \, , \, x\in \R^d \, ,\\
u(0,x)=\phi(x) \, .
\end{array}
\right.
\end{equation*}
Then, it holds that 
\begin{equation*}
\|u\|_{L^{\frac{1}{\eta}}([0,T]; H^{-s+\eta}_\rho(\mathbb{R}^d))}\lesssim \|\phi\|_{H^{-s}(\mathbb{R}^d)}+\|F\|_{L^1([0,T]; H^{-s}(\mathbb{R}^d))}\, ,
\end{equation*}
where the proportional constant only depends on $\rho$, $s$ and $\eta$.
\end{lemma}

Let us go back to equation (\ref{ccl-snls}). We see that, compared with \cite{Schaeffer1}, the main difference lies in the fact that the worst term between $\displaystyle -i\int_0^t e^{-i(t-\tau)\Delta}(\luxor_{\tau})\, d\tau$ and $\cherry$ is now $\displaystyle -i\int_0^t e^{-i(t-\tau)\Delta}(\luxor_{\tau})\, d\tau$ with its $-s+\eta$ derivatives (according to the previous local smoothing). Consequently, here, $v$ is expected to inherit the regularity of $\displaystyle -i\int_0^t e^{-i(t-\tau)\Delta}(\luxor_{\tau})\, d\tau$ and should be an element of $\mathcal{C}([0,T];H^{-s+\eta}(\mathbb{R}^d))$. Precisely, by resorting to the local smoothing of the Schr\"{o}dinger operator (see Lemma \ref{local-regu}), we will be able to prove that, up to multiplication by $\rho$, $v \in \mathcal{C}([0,T];\mathcal{W}^{-s+\eta,p}(\mathbb{R}^d))$ (for some $p \geq 2$) with $\eta>0$ such that $-s+\eta>0$. Thus, for all $0 \leq t \leq T$, $v(t)$ will be locally a function, allowing us to give a rigorous meaning to $\rho^2|v|^2$.

\subsubsection{About the product in Sobolev spaces}

The following lemma states that the product of two functions of regularity $s \geq 0$ stays a function of regularity $s \geq 0$ and precises its integrability under an H\"{o}lder type condition. See for instance {\cite[Proposition 1.1, p. 105]{Taylor-ter}}.
\begin{lemma}[Fractional Leibniz rule]\label{frac-leibniz}
Let $s\geq 0$, $1<r<\infty$ and $1<p_1,p_2,q_1,q_2< \infty$ satisfying 
$$\frac{1}{r}=\frac{1}{p_1}+\frac{1}{p_2}=\frac{1}{q_1}+\frac{1}{q_2} \, .$$ 
Then, it holds that
$$\|u\cdot v\|_{\mathcal{W}^{s,r}(\mathbb{R}^d)}\lesssim \|u\|_{\mathcal{W}^{s,p_1}(\mathbb{R}^d)}\|v\|_{L^{p_2}(\mathbb{R}^d)}+\|u\|_{L^{q_1}(\mathbb{R}^d)}\|v\|_{\mathcal{W}^{s,q_2}(\mathbb{R}^d)} \, .$$
\end{lemma}
\noindent
Now, we are interested in defining the product $f \cdot g$ when $g$ is a function (of regularity $\beta>0$) and $f$ is only a distribution of Sobolev regularity $-\alpha<0$. This is the subject of the lemma below that guarantees that the product makes sense as soon as $\beta> \alpha$ and inherits the worst regularity, namely $-\alpha$. The proof of this result can be found in \cite[Section 4.4.3]{prod-ter}.
\begin{lemma}\label{product}
Fix $d\geq 1$ a space dimension. Let $\al,\be >0$ and {$1 \leq p,p_1,p_2< \infty$} be such that {
$$
\frac{1}{p} = \frac{1}{p_{1}}+\frac{1}{p_{2}}\quad \text{and} \quad 0<\alpha<\beta.
$$}
If $f\in \cw^{-\al,p_1}(\R^d)$ and $g\in \cw^{\be,p_2}(\R^d)$, then $f\cdot g\in \cw^{-\al,p}(\R^d)$ and the following bound holds true
$$
\| f\cdot g\|_{\cw^{-\al,p}(\mathbb{R}^d)} \lesssim \|f\|_{\cw^{-\al,p_1}(\mathbb{R}^d)} \| g\|_{\cw^{\be,p_2}(\mathbb{R}^d)} \ .$$
\end{lemma}

\subsubsection{An interpolation result and a commutator estimate}
To end with, we recall a classical interpolation result followed by a commutator estimate proven in \cite{Schaeffer1} that permits to swap a $\mathcal{C}^{\infty}$ compactly-supported function with the fractional Laplacian.
\begin{lemma}\label{interpol}
Fix $d\geq 1$ a space dimension. Let $s,s_1,s_2\in \R$ and $1\leq p,p_1,p_2<\infty$. Suppose that there exists $\theta \in (0,1)$ such that
$$s=\theta s_1+(1-\theta)s_2 \quad \text{and} \quad \frac{1}{p}=\frac{\theta}{p_1}+\frac{1-\theta}{p_2} \, .$$
If $v\in \cw^{s_1,p_1}(\R^d)\cap \cw^{s_2,p_2}(\R^d)$, then $v\in \cw^{s,p}(\R^d)$ and it holds that
$$\|v\|_{\cw^{s,p}(\mathbb{R}^d)}\leq \|v\|_{\cw^{s_1,p_1}(\mathbb{R}^d)}^\theta \|v\|_{\cw^{s_2,p_2}(\mathbb{R}^d)}^{1-\theta}\, .$$
\end{lemma}

\begin{lemma}\label{commut}
For every $s >0$ and for all $\mathcal{C}^{\infty}$ compactly-supported functions $\rho, g:\R^d \to \R$, it holds that
\begin{equation*}
 \|(\emph{\id}-\Delta)^{\frac{s}{2}}(\rho\cdot g)-\rho \cdot (\emph{\id}-\Delta)^{\frac{s}{2}}(g)\|_{L^2(\mathbb{R}^d)}\lesssim \|g\|_{H^{s-1}(\mathbb{R}^d)} \, ,
\end{equation*}
where the proportional constant only depends on $\rho$ and $s$.
\end{lemma}

\subsection{Statement and proof of our main result}

\noindent
Let us fix once and for all a $\mathcal{C}^{\infty}$ compactly-supported function $\rho:\R^d\to\R$ of the form~\eqref{rho}. For all $T\geq 0$, $\eta>0$, $p,q\geq 2$, let $Y^{s,\eta,(p,q)}_\rho(T)$ be the space defined by 
\begin{equation*}
Y^{s,\eta,(p,q)}_\rho(T):=\mathcal{C}([0,T]; H^{-s}(\mathbb{R}^d))\cap L^p([0,T]; \cw^{-s,q}(\mathbb{R}^d))\cap L^{\frac{1}{\eta}}_T H^{-s+\eta}_\rho \, .
\end{equation*}

\begin{theorem}
Let $1\leq d\leq 3$ be a space dimension and $\rho: \mathbb{R}^d \rightarrow \mathbb{R}$ be a $\mathcal{C}^{\infty}$ compactly-supported function of the form \eqref{rho}. Besides, assume that $\displaystyle\alpha<\frac{d}{2}$ verifies $\alpha> \alpha_d$, where
\begin{equation}\label{condition-alpha}
\alpha_d=
\left\{
\begin{array}{l}
1/4 \quad \text{if} \ d=1\\
5/6 \quad \text{if} \ d=2\\
17/12 \quad \text{if} \ d=3 
\end{array}
\right. \, .
\end{equation}
Recall that, for all fixed $\varepsilon>0$, the real number $s_{d,\alpha,\varepsilon}>0$ is defined by $s=s_{d,\alpha,\varepsilon}=(\frac{d}{2}-\alpha)+\varepsilon$.
Then, if $\varepsilon>0$ is small enough, one can find parameters ${\eta\in [s,1/2]}$ and $p,q\geq 2$ such that, for all $\phi \in H^{-s}(\mathbb{R}^d)$ and $(\luxor,\cherry) \in \mathcal{E}_{s,\kappa}$, there exists a time $T>0$ for which equation~\eqref{ccl-snls} admits a unique solution in the above-defined set $Y^{s,\eta,(p,q)}_\rho(T)$.
\end{theorem}
\noindent
Our strategy is to establish a fixed-point principle on $\Gamma_{T, \luxor,\cherry}$ the map defined for all $T\geq 0$ and $(\luxor,\cherry)\in \mathcal{E}_{s,\kappa}$ by
\begin{multline*} 
\Gamma_{T, \luxor,\cherry}(v):=e^{-it\Delta}\phi-i\int_0^t e^{-i(t-\tau)\Delta}(\rho^2 |v_{\tau}|^2)\, d\tau-i\int_0^t e^{-i(t-\tau)\Delta}((\rho \overline{v}_{\tau})\cdot(\luxor_{\tau}))\, d\tau\\
-i\int_0^t e^{-i(t-\tau)\Delta}((\rho v_{\tau})\cdot(\overline{\luxor_{\tau}}))\, d\tau+\cherry , \quad t\in [0,T] \, .
\end{multline*}
This will result from the two bounds described in the proposition below.

\begin{proposition}\label{technical}
Assume that $1\leq d\leq 3$ and that $\alpha_d$ satisfies condition~\eqref{condition-alpha}. Then,  if $\varepsilon>0$ is small enough, one can find parameters $\eta>0$, $p,q\geq 2$ and $\tilde{\varepsilon} >0$ such that, setting $Y(T):=Y^{s,\eta,(p,q)}_\rho(T)$, the following bounds hold true: for all $0\leq T\leq 1$, $\phi \in H^{-s}(\mathbb{R}^d), (\luxor_1,\cherry_1) \in \mathcal{E}_{s,\kappa},  (\luxor_2,\cherry_2) \in \mathcal{E}_{s,\kappa}$ and $v, v_1, v_2 \in Y(T)$,
\begin{equation}
\|\Gamma_{T, \luxor_1,\cherry_1}(v)\|_{Y(T)}\lesssim  \|\phi\|_{H^{-s}}+T^{\tilde{\varepsilon}}\Big[\|v\|_{Y(T)}^2+\|\luxor_1\|_{L^\infty_T {\mathcal{W}^{-s,r}}} \|v\|_{Y(T)}\Big]+\|\cherry_1\|_{Y(T)} , \label{ctrl1SNLS}
\end{equation}
and
\begin{align}
&\|\Gamma_{T, \luxor_1, \cherry_1}(v_1)-\Gamma_{T, \luxor_2, \cherry_2}(v_2)\|_{Y(T)}
\nonumber \\
&\lesssim \ T^{\tilde{\varepsilon}}\Big[\|v_1-v_2\|_{Y(T)}\{\|v_1\|_{Y(T)}+\|v_2\|_{Y(T)}\}+\|\luxor_1-\luxor_2\|_{L^\infty_T {\mathcal{W}^{-s,r}}}\|v_1\|_{Y(T)}\nonumber\\
&\hspace{4cm}+\|\luxor_2\|_{L^\infty_T {\mathcal{W}^{-s,r}}}\| v_1-v_2\|_{Y(T)}\Big]+\|\cherry_1-\cherry_2\|_{Y(T)}\label{ctrl2SNLS} \, , 
\end{align}
where $r$ depends on $s$ and $\eta$ and the proportional constants depend only on $\rho$ and $s$.
\end{proposition}
\noindent
The choice of the three parameters $\eta,p,q$ in the above proposition depends on the space dimension $d\in \{1,2,3\}$. For the sake of clarity, we divide our proof into three subcases. As usual, our objective is to bound each term in the expression of $\Gamma_{T, \luxor, \cherry}$ separately. In the following, we will suppose that $0 \leq T \leq 1$.

\subsubsection{Proof of Proposition~\ref{technical} when $d=1$}\label{proof-dim-1}
As soon as $\varepsilon>0$ is small enough, there exists $\eta>0$ such that $$2s<\eta <\inf(\frac12,\frac34-s).$$ Let $(p,q)=(\infty,2).$ Then,
$$Y(T)=\mathcal{C}([0,T]; H^{-s}(\mathbb{R}))\cap  L^{\frac{1}{\eta}}_T H^{-s+\eta}_\rho \, .$$
Also, we define $\theta=\frac{s}{\eta}\in (0,\frac12)$. Notice that a quick computation shows that the pair $(p,q)$ is Schr{\"o}dinger admissible.

\noindent
\textbf{Bound on $e^{-it\Delta}\phi$:}
As $e^{-it\Delta}$ is a unitary operator on $H^{-s}(\mathbb{R})$, it holds that
$$\|e^{-it\Delta}\phi\|_{L^{\infty}_T H^{-s}}= \|\phi\|_{H^{-s}(\mathbb{R})}\, .$$
Besides, since $s\leq \frac12$ and $\eta\leq \frac12$, by local regularization (see Lemma~\ref{local-regu}),
$$\|e^{-it\Delta}\phi\|_{L^{\frac{1}{\eta}}_T H^{-s+\eta}_\rho}\lesssim  \|\phi\|_{H^{-s}(\mathbb{R})}\, ,$$
and, combining the two previous inequalities, we have established that
$$\|e^{-it\Delta}\phi\|_{Y(T)}\lesssim \|\phi\|_{H^{-s}(\mathbb{R})}.$$

\noindent
\textbf{Bound on $\displaystyle-i\int_0^t e^{-i(t-\tau)\Delta}(\rho^2 |v_{\tau}|^2)\, d\tau$:}
Since $\eta >0$ and since $\rho$ is a $\mathcal{C}^{\infty}$ compactly-supported function, one has
\begin{align*}
&\bigg\|-i\int_0^t e^{-i(t-\tau)\Delta}(\rho^2 |v_{\tau}|^2)\, d\tau\bigg\|_{Y(T)}=\\
&\bigg\|-i\int_0^t e^{-i(t-\tau)\Delta}(\rho^2 |v_{\tau}|^2)\, d\tau\bigg\|_{L^{\infty}_T H^{-s}}+\bigg\|-i\int_0^t e^{-i(t-\tau)\Delta}(\rho^2 |v_{\tau}|^2)\, d\tau\bigg\|_{L^{\frac{1}{\eta}}_T H^{-s+\eta}_\rho}\\
&\lesssim \bigg\|-i\int_0^t e^{-i(t-\tau)\Delta}(\rho^2 |v_{\tau}|^2)\, d\tau\bigg\|_{L^{\infty}_T H^{-s+\eta}}+\bigg\|-i\int_0^t e^{-i(t-\tau)\Delta}(\rho^2 |v_{\tau}|^2)\, d\tau\bigg\|_{L^{\frac{1}{\eta}}_T H^{-s+\eta}}\\
& \lesssim \bigg\| -i\int_0^t e^{-i(t-\tau)\Delta}(\rho^2 |v_{\tau}|^2)\, d\tau\bigg\|_{L^{\infty}_T H^{-s+\eta}}\, .\label{refer-proof-dim-1}
\end{align*}
Resorting to Strichartz inequalities (Lemma~\ref{lem:strichartz2}), it holds that
\begin{equation}\label{schro-dim-1}
\bigg\|-i\int_0^t e^{-i(t-\tau)\Delta}(\rho^2 |v_{\tau}|^2)\, d\tau\bigg\|_{Y(T)}\lesssim \|\rho^2 |v|^2\|_{L^{\frac43}_T \mathcal{W}^{-s+\eta,1}} \, .
\end{equation}
Now, according to Leibniz fractional rule (Lemma~\ref{frac-leibniz}), for all $t\geq 0$,
$$
\|\rho^2 |v|^2(t,.)\|_{\mathcal{W}^{-s+\eta,1}}\lesssim \| \rho v(t,.)\|_{H^{-s+\eta}}\| \rho v(t,.)\|_{L^2}\, ,
$$
and, by interpolation (Lemma~\ref{interpol}),
$$\| \rho v(t,.)\|_{L^2}\leq  \| \rho v(t,.)\|_{H^{-s+\eta}}^\theta \| \rho v(t,.)\|_{H^{-s}}^{1-\theta} \, ,$$
leading to
\begin{align*}
\|\rho^2 |v|^2(t,.)\|_{\mathcal{W}^{-s+\eta,1}}&\lesssim \| \rho v(t,.)\|_{H^{-s+\eta}}^{1+\theta}\| \rho v(t,.)\|_{H^{-s}}^{1-\theta} \\
&\lesssim \|  v(t,.)\|_{H^{-s+\eta}_\rho}^{1+\theta}\| v(t,.)\|_{H^{-s}}^{1-\theta}+\| v(t,.)\|_{H^{-s}}^{2}   \, ,
\end{align*}
where we have used  a commutator estimate (Lemma~\ref{commut}) to obtain the second inequality. Consequently,
\begin{align*}
&\int_0^T dt \, \|\rho^2 |v|^2(t,.)\|_{\mathcal{W}^{-s+\eta,1}}^{\frac43}\\
&\lesssim \|v\|_{Y(T)}^{\frac43(1-\theta)}  \int_0^T dt \, \|  v(t,.)\|_{H^{-s+\eta}_\rho}^{\frac43(1+\theta)} +T \|v\|_{X(T)}^{\frac83}  \\
&\lesssim T^{1-\frac43(1+\theta)\eta}\|v\|_{Y(T)}^{\frac43(1-\theta)}  \bigg(\int_0^T dt \, \|  v(t,.)\|_{H^{-s+\eta}_\rho}^{\frac{1}{\eta}}\bigg)^{\frac43(1+\theta)\eta} +T \|v\|_{Y(T)}^{\frac83}\\
&\lesssim  T^{1-\frac43(1+\theta)\eta} \|v\|_{Y(T)}^{\frac83} \, ,
\end{align*}
and thus, coming back to~\eqref{schro-dim-1}, we can conclude that
\begin{equation*} 
\bigg\|-i\int_0^t e^{-i(t-\tau)\Delta}(\rho^2 |v_{\tau}|^2)\, d\tau\bigg\|_{Y(T)} \lesssim  T^{\frac34-(\eta+s)} \|v\|_{Y(T)}^{2} \, .
\end{equation*}

\noindent
\textbf{Bound on $\displaystyle-i\int_0^t e^{-i(t-\tau)\Delta}((\rho \overline{v}_{\tau})\cdot(\luxor_{\tau}))\, d\tau$, $\displaystyle-i\int_0^t e^{-i(t-\tau)\Delta}((\rho v_{\tau})\cdot(\overline{\luxor_{\tau}}))\, d\tau$:}
Since $s\leq \frac12$ and $\eta\leq \frac12$, by local regularization (Lemma~\ref{local-regu}), the following bound holds true
\begin{align}
&\bigg\|-i\int_0^t e^{-i(t-\tau)\Delta}((\rho \overline{v}_{\tau})\cdot(\luxor_{\tau}))\, d\tau\bigg\|_{L^{\frac{1}{\eta}}_T H^{-s+\eta}_\rho}+\bigg\|-i\int_0^t e^{-i(t-\tau)\Delta}((\rho v_{\tau})\cdot(\overline{\luxor_{\tau}}))\, d\tau\bigg\|_{L^{\frac{1}{\eta}}_T H^{-s+\eta}_\rho}\nonumber\\
&\lesssim \|\rho \overline{v} \cdot \luxor\|_{L^1_T H^{-s}}\, .\label{pro}
\end{align}
Let $l>s$ and $2 \leq r,p<\infty$ be such that $1/2=1/r+1/p$. According to Lemma~\ref{product}, for all $t\geq 0$,
\begin{equation}\label{injection}
\|(\rho \overline{v} \cdot \luxor)(t,.)\|_{H^{-s}} \lesssim\| \luxor(t,.)\|_{\mathcal{W}^{-s,r}}\| \rho v(t,.)\|_{\mathcal{W}^{l,p}} \, .
\end{equation}
As soon as 
$$-s+\eta-l=d(\frac12-\frac1p)=\frac{d}r \quad \Leftrightarrow \quad l=-s +\eta -\frac{d}r\;, $$
we have the Sobolev embedding $
H^{-s +\eta}(\mathbb{R}^d)\hookrightarrow \mathcal{W}^{l,p}(\mathbb{R}^d)$. Now, as $-s +\eta >s$, we can pick $r\geq 2$ large enough such that $l=-s +\eta -\frac{d}r>s$, and from \eqref{injection}, we get that 
$$
\|(\rho \overline{v} \cdot \luxor)(t,.)\|_{H^{-s}} \lesssim\| \luxor(t,.)\|_{\mathcal{W}^{-s,r}}\| \rho v(t,.)\|_{H^{-s+\eta}} \, .
$$
With the help of a commutator estimate (Lemma~\ref{commut}), for all $t\geq 0$,
\begin{equation*}
\| \rho v(t,.)\|_{H^{-s+\eta}}\lesssim \|v(t,.)\|_{H^{-s+\eta}_\rho}+\|v(t,.)\|_{H^{-s}}\, ,
\end{equation*}
entailing that
\begin{align*}
\|\rho \overline{v} \cdot \luxor\|_{L^1_T H^{-s}}&\lesssim \|\luxor\|_{L^\infty_T {\mathcal{W}^{-s,r}}}\{T^{1-\eta}\|v\|_{L^{\frac{1}{\eta}}_T H^{-s+\eta}_\rho}+T\, \|v\|_{L^\infty_T H^{-s}}\}\\
&\lesssim T^{1-\eta}  \|\luxor\|_{L^\infty_T {\mathcal{W}^{-s,r}}}\|v\|_{Y(T)}\, ,
\end{align*}
which, going back to~\eqref{pro}, leads us to
\begin{align}\label{bound-pro}
&\bigg\|-i\int_0^t e^{-i(t-\tau)\Delta}((\rho \overline{v}_{\tau})\cdot(\luxor_{\tau}))\, d\tau\bigg\|_{L^{\frac{1}{\eta}}_T H^{-s+\eta}_\rho}+\bigg\|-i\int_0^t e^{-i(t-\tau)\Delta}((\rho v_{\tau})\cdot(\overline{\luxor_{\tau}}))\, d\tau\bigg\|_{L^{\frac{1}{\eta}}_T H^{-s+\eta}_\rho}\nonumber\\
&\lesssim T^{1-\eta}  \|\luxor\|_{L^\infty_T {\mathcal{W}^{-s,r}}}\|v\|_{Y(T)} \, .
\end{align}
\noindent
On the other hand, resorting to Strichartz inequalities (Lemma~\ref{lem:strichartz2}), it holds that
\begin{equation*}
\bigg\|-i\int_0^t e^{-i(t-\tau)\Delta}((\rho \overline{v}_{\tau})\cdot(\luxor_{\tau}))\, d\tau\bigg\|_{L^\infty_T H^{-s}}+\bigg\|-i\int_0^t e^{-i(t-\tau)\Delta}((\rho v_{\tau})\cdot(\overline{\luxor_{\tau}}))\, d\tau\bigg\|_{L^\infty_T H^{-s}}  
\lesssim \|\rho \overline{v} \cdot \luxor\|_{L^1_T H^{-s}}\, .
\end{equation*}
We are thus in the same position as in~\eqref{pro}, and we can repeat our arguments to establish that
\begin{align*}
&\bigg\|-i\int_0^t e^{-i(t-\tau)\Delta}((\rho \overline{v}_{\tau})\cdot(\luxor_{\tau}))\, d\tau\bigg\|_{L^\infty_T H^{-s}}+\bigg\|-i\int_0^t e^{-i(t-\tau)\Delta}((\rho v_{\tau})\cdot(\overline{\luxor_{\tau}}))\, d\tau\bigg\|_{L^\infty_T H^{-s}}\\
&\lesssim T^{1-\eta} \|\luxor\|_{L^\infty_T {\mathcal{W}^{-s,r}}} \|v\|_{Y(T)} \, .
\end{align*}
\noindent
Combining this estimate with~\eqref{bound-pro}, we have finally obtained that
$$\bigg\|-i\int_0^t e^{-i(t-\tau)\Delta}((\rho \overline{v}_{\tau})\cdot(\luxor_{\tau}))\, d\tau\bigg\|_{Y(T)}+\bigg\|-i\int_0^t e^{-i(t-\tau)\Delta}((\rho v_{\tau})\cdot(\overline{\luxor_{\tau}}))\, d\tau\bigg\|_{Y(T)}\lesssim T^{1-\eta}  \|\luxor\|_{L^\infty_T {\mathcal{W}^{-s,r}}} \|v\|_{Y(T)} \, .$$

\noindent
\textbf{Bound on $\cherry_1$:}
Keeping in mind that $\kappa=1-\alpha$ and $\eta < \frac{1}{2}$, $-s+\eta\leq -2s + \kappa$ if $\varepsilon>0$ is small enough and we immediately have that
$$\|\cherry_1\|_{Y(T)}\leq \|\cherry_1\|_{L^{\infty}_T H^{-2s+\kappa}} .$$
\noindent
Combining the above bounds provides us with~\eqref{ctrl1SNLS}. \eqref{ctrl2SNLS} can easily be obtained in an analogous manner.

\subsubsection{Proof of Proposition~\ref{technical} when $d=2$}
As soon as $\varepsilon>0$ is small enough, there exists $\eta > 0$ such that $$2s<\eta <\frac12-s.$$ Let $(p,q)=(4,4)$. Then,
$$Y(T)=\mathcal{C}([0,T]; H^{-s}(\mathbb{R}^2))\cap L^4([0,T]; \cw^{-s,4}(\mathbb{R}^2))\cap L^{\frac{1}{\eta}}_T H^{-s+\eta}_\rho \, .$$
Also, we define $\theta=\frac{s}{\eta}\in (0,\frac12)$. Notice that a quick computation shows that the pair $(p,q)$ is Schr{\"o}dinger admissible.

\noindent
\textbf{Bound on $e^{-it\Delta}\phi$:}
Exactly as for $d=1$ (see Section~\ref{proof-dim-1}), we show that
$$\|e^{-it\Delta}\phi\|_{Y(T)}\lesssim \|\phi\|_{H^{-s}(\mathbb{R}^2)}.$$

\noindent
\textbf{Bound on $\displaystyle-i\int_0^t e^{-i(t-\tau)\Delta}(\rho^2 |v_{\tau}|^2)\, d\tau$:}
Since $\eta >0$ and since $\rho$ is a $\mathcal{C}^{\infty}$ compactly-supported function, one has
\begin{align*}
&\bigg\|-i\int_0^t e^{-i(t-\tau)\Delta}(\rho^2 |v_{\tau}|^2)\, d\tau\bigg\|_{Y(T)}=\bigg\|-i\int_0^t e^{-i(t-\tau)\Delta}(\rho^2 |v_{\tau}|^2)\, d\tau\bigg\|_{L^{\infty}_T H^{-s}}\\
&+\bigg\|-i\int_0^t e^{-i(t-\tau)\Delta}(\rho^2 |v_{\tau}|^2)\, d\tau\bigg\|_{L^{4}_T \cw^{-s,4}}+\bigg\|-i\int_0^t e^{-i(t-\tau)\Delta}(\rho^2 |v_{\tau}|^2)\, d\tau\bigg\|_{L^{\frac{1}{\eta}}_T H^{-s+\eta}_\rho}\\
&\lesssim \bigg\|-i\int_0^t e^{-i(t-\tau)\Delta}(\rho^2 |v_{\tau}|^2)\, d\tau\bigg\|_{L^{\infty}_T H^{-s+\eta}}+\bigg\|-i\int_0^t e^{-i(t-\tau)\Delta}(\rho^2 |v_{\tau}|^2)\, d\tau\bigg\|_{L^{4}_T \cw^{-s+\eta,4}}\\
&\hspace{0,5cm}+\bigg\|-i\int_0^t e^{-i(t-\tau)\Delta}(\rho^2 |v_{\tau}|^2)\, d\tau\bigg\|_{L^{\frac{1}{\eta}}_T H^{-s+\eta}} \\
&\lesssim \bigg\| -i\int_0^t e^{-i(t-\tau)\Delta}(\rho^2 |v_{\tau}|^2)\, d\tau\bigg\|_{L^{\infty}_T H^{-s+\eta}}+\bigg\|-i\int_0^t e^{-i(t-\tau)\Delta}(\rho^2 |v_{\tau}|^2)\, d\tau\bigg\|_{L^{4}_T \cw^{-s+\eta,4}}\, ,
\end{align*}
and, resorting to Strichartz inequalities (Lemma~\ref{lem:strichartz2}), we deduce that
\begin{equation}\label{schro-dim-2}
\bigg\|-i\int_0^t e^{-i(t-\tau)\Delta}(\rho^2 |v_{\tau}|^2)\, d\tau\bigg\|_{Y(T)}\lesssim \|\rho^2 |v|^2\|_{L^{r'}_T \mathcal{W}^{-s+\eta,l'}} \,
\end{equation}
where
$$(r,l)=(\frac{4}{1+\theta},\frac{4}{1-\theta})\, $$
is a Schr{\"o}dinger admissible pair. Now, according to Leibniz fractional rule (Lemma~\ref{frac-leibniz}), for all $t\geq 0$,
$$
\|\rho^2 |v|^2(t,.)\|_{\mathcal{W}^{-s+\eta,l'}}\lesssim \| \rho v(t,.)\|_{H^{-s+\eta}}\| \rho v(t,.)\|_{L^{\frac{4}{1+\theta}}}\, .
$$
Then, by interpolation (see Lemma~\ref{interpol}), the bound below holds true
$$\| \rho v(t,.)\|_{L^{\frac{4}{1+\theta}}}\leq  \| \rho v(t,.)\|_{H^{-s+\eta}}^\theta \| \rho v(t,.)\|_{\cw^{-s,4}}^{1-\theta} \, $$
and, consequently,
\begin{align*}
\|\rho^2 |v|^2(t,.)\|_{\mathcal{W}^{-s+\eta,l'}}&\lesssim \| \rho v(t,.)\|_{H^{-s+\eta}}^{1+\theta}\| \rho v(t,.)\|_{\cw^{-s,4}}^{1-\theta} \\
&\lesssim \|  v(t,.)\|_{H^{-s+\eta}_\rho}^{1+\theta}\| v(t,.)\|_{\cw^{-s,4}}^{1-\theta}+\| v(t,.)\|_{H^{-s}}^{1+\theta}\| v(t,.)\|_{\cw^{-s,4}}^{1-\theta}   \, ,
\end{align*}
where we have used a commutator estimate (Lemma~\ref{commut}) to obtain the second inequality. H{\"o}lder's inequality with $\la=\frac{3-\theta}{2} > 1$ yields
\begin{align*}
&\int_0^T dt \, \|\rho^2 |v|^2(t,.)\|_{\mathcal{W}^{-s+\eta,l'}}^{r'}\\
&\lesssim \bigg( \int_0^T dt \, \|  v(t,.)\|_{H^{-s+\eta}_\rho}^{(1+\theta)r'\la}\bigg)^{\frac{1}{\la}} \bigg(\int_0^T dt \, \| v(t,.)\|_{\cw^{-s,4}}^{(1-\theta)r'\la'}\bigg)^{\frac{1}{\la'}}+\|v\|_{Y(T)}^{r'(1+\theta)} \int_0^T dt \, \|v_t\|^{r'(1-\theta)}_{\cw^{-s,4}} \, .
\end{align*}
According to the assumptions on our parameters (recall that $\eta <\frac12-s$ and $\theta=\frac{s}{\eta}$),
$$(1+\theta)r'\la=2(1+\theta)<\frac{1}{\eta} \quad \text{and} \quad (1-\theta)r'\la' =4 \, ,$$
yielding
\begin{align*}
&\int_0^T dt \, \|\rho^2 |v|^2(t,.)\|_{\mathcal{W}^{-s+\eta,l'}}^{r'}\\
&\lesssim T^{\frac{1-2\eta(1+\theta)}{\la}}\bigg( \int_0^T dt \, \|  v(t,.)\|_{H^{-s+\eta}_\rho}^{\frac{1}{\eta}}\bigg)^{(1+\theta)r' \eta} \|v\|_{Y(T)}^{\frac{4}{\la'}}+T^{1-\frac{r'(1-\theta)}{4}}\|v\|_{Y(T)}^{r'(1+\theta)} \|v\|_{Y(T)}^{r'(1-\theta)} \\
&\lesssim T^{\frac{1-2\eta-2s}{\la}}\|v\|_{Y(T)}^{(1+\theta)r'+\frac{4}{\la'}}+T^{1-\frac{r'(1-\theta)}{4}}\|v\|_{Y(T)}^{2r'}  \, .
\end{align*}
The previous estimate can be reformulated as
$$
\|\rho^2 |v|^2\|_{L^{r'}_T \mathcal{W}^{-s+\eta,l'}}  \lesssim \{T^{\tilde{\varepsilon}}+T^{\frac12}\} \|v\|_{Y(T)}^2 \, ,
$$
with $\tilde{\varepsilon}=\frac{1}{2} (1-2\eta-2s)$.
Coming back to~\eqref{schro-dim-2}, we have finally obtained that
\begin{equation*}
\bigg\|-i\int_0^t e^{-i(t-\tau)\Delta}(\rho^2 |v_{\tau}|^2)\, d\tau\bigg\|_{Y(T)} \lesssim \{T^{\tilde{\varepsilon}}+T^{\frac12}\} \|v\|_{Y(T)}^2 \, .
\end{equation*}

\noindent
\textbf{Bound on $\displaystyle-i\int_0^t e^{-i(t-\tau)\Delta}((\rho \overline{v}_{\tau})\cdot(\luxor_{\tau}))\, d\tau$, $\displaystyle-i\int_0^t e^{-i(t-\tau)\Delta}((\rho v_{\tau})\cdot(\overline{\luxor_{\tau}}))\, d\tau$:}
Repeating the arguments used for $d=1$, we first establish that
\begin{equation*}
\bigg\|-i\int_0^t e^{-i(t-\tau)\Delta}((\rho \overline{v}_{\tau})\cdot(\luxor_{\tau}))\, d\tau\bigg\|_{L^{\frac{1}{\eta}}_T H^{-s+\eta}_\rho} +\bigg\|-i\int_0^t e^{-i(t-\tau)\Delta}((\rho v_{\tau})\cdot(\overline{\luxor_{\tau}}))\, d\tau\bigg\|_{L^{\frac{1}{\eta}}_T H^{-s+\eta}_\rho}\lesssim \|\rho \overline{v} \cdot \luxor\|_{L^1_T H^{-s}} \, ,
\end{equation*}
and then
\begin{equation*}
\|\rho \overline{v} \cdot \luxor\|_{L^1_T H^{-s}} \lesssim T^{1-\eta}    \|\luxor\|_{L^\infty_T {\mathcal{W}^{-s,r}}} \|v\|_{Y(T)}\, .
\end{equation*}
On the other hand, resorting to Strichartz inequalities (Lemma~\ref{lem:strichartz2}), we can write
\begin{align*}
&\bigg(\bigg\|-i\int_0^t e^{-i(t-\tau)\Delta}((\rho \overline{v}_{\tau})\cdot(\luxor_{\tau}))\, d\tau\bigg\|_{L^\infty_T H^{-s}}+\bigg\|-i\int_0^t e^{-i(t-\tau)\Delta}((\rho v_{\tau})\cdot(\overline{\luxor_{\tau}}))\, d\tau\bigg\|_{L^\infty_T H^{-s}}\bigg)\nonumber\\
&+\bigg(\bigg\|-i\int_0^t e^{-i(t-\tau)\Delta}((\rho \overline{v}_{\tau})\cdot(\luxor_{\tau}))\, d\tau\bigg\|_{L^4_T \cw^{-s,4}}+\bigg\|-i\int_0^t e^{-i(t-\tau)\Delta}((\rho v_{\tau})\cdot(\overline{\luxor_{\tau}}))\, d\tau\bigg\|_{L^4_T \cw^{-s,4}}\bigg)\\
&\lesssim \|\rho \overline{v} \cdot \luxor\|_{L^1_T H^{-s}} \, .
\end{align*}
Combining the two previous bounds, we obtain that
$$\bigg\|-i\int_0^t e^{-i(t-\tau)\Delta}((\rho \overline{v}_{\tau})\cdot(\luxor_{\tau}))\, d\tau\bigg\|_{Y(T)}+\bigg\|-i\int_0^t e^{-i(t-\tau)\Delta}((\rho v_{\tau})\cdot(\overline{\luxor_{\tau}}))\, d\tau\bigg\|_{Y(T)}\lesssim T^{1-\eta}\|\luxor\|_{L^\infty_T {\mathcal{W}^{-s,r}}} \|v\|_{Y(T)} \, .$$

\noindent
\textbf{Bound on $\cherry_1$:}
Keeping in mind that $\kappa=\frac{3}{2}-\alpha$ and $\eta < \frac{1}{2}$, $-s+\eta\leq -2s + \kappa$ if $\varepsilon>0$ is small enough and we immediately have that
$$\|\cherry_1\|_{Y(T)}\leq \|\cherry_1\|_{L^{\infty}_T H^{-2s+\kappa}} + \|\cherry_1\|_{L^{\infty}_T \mathcal{W}^{-2s+\kappa,4}}.$$

\

\subsubsection{Proof of Proposition~\ref{technical} when $d=3$}
As soon as $\varepsilon>0$ is small enough, there exists $\eta>0$ such that $$2s<\eta <\frac14-s.$$ Let $(p,q)=(2,6)$. Then,
$$Y(T)=\mathcal{C}([0,T]; H^{-s}(\mathbb{R}^3))\cap L^2([0,T]; \cw^{-s,6}(\mathbb{R}^3))\cap L^{\frac{1}{\eta}}_T H^{-s+\eta}_\rho \, .$$
Also, we define $\theta=\frac{s}{\eta}\in (0,\frac12)$. Notice that a quick computation shows that the pair $(p,q)$ is Schr{\"o}dinger admissible.

\noindent
\textbf{Bound on $e^{-it\Delta}\phi$:}
Again, the arguments are those used for $d=1$ and $d=2$ yielding
$$\|e^{-it\Delta}\phi\|_{Y(T)}\lesssim \|\phi\|_{H^{-s}(\mathbb{R}^3)}.$$

\noindent
\textbf{Bound on $\displaystyle -i\int_0^t e^{-i(t-\tau)\Delta}(\rho^2 |v_{\tau}|^2)\, d\tau$:}
By definition,
\begin{align*}
&\bigg\|-i\int_0^t e^{-i(t-\tau)\Delta}(\rho^2 |v_{\tau}|^2)\, d\tau\bigg\|_{Y(T)}=\bigg\|-i\int_0^t e^{-i(t-\tau)\Delta}(\rho^2 |v_{\tau}|^2)\, d\tau\bigg\|_{L^{\infty}_T H^{-s}}\\
&+\bigg\|-i\int_0^t e^{-i(t-\tau)\Delta}(\rho^2 |v_{\tau}|^2)\, d\tau\bigg\|_{L^{2}_T \cw^{-s,6}}+\bigg\|-i\int_0^t e^{-i(t-\tau)\Delta}(\rho^2 |v_{\tau}|^2)\, d\tau\bigg\|_{L^{\frac{1}{\eta}}_T H^{-s+\eta}_\rho}\\
&\lesssim \bigg\|-i\int_0^t e^{-i(t-\tau)\Delta}(\rho^2 |v_{\tau}|^2)\, d\tau\bigg\|_{L^{\infty}_T H^{-s+\eta}}+\bigg\|-i\int_0^t e^{-i(t-\tau)\Delta}(\rho^2 |v_{\tau}|^2)\, d\tau\bigg\|_{L^{2}_T \cw^{-s+\eta,6}}\, ,
\end{align*}
since $\eta >0$ and since $\rho$ is a $\mathcal{C}^{\infty}$ compactly-supported function. Now, resorting to Strichartz inequalities (Lemma~\ref{lem:strichartz2}), we deduce that
\begin{equation}\label{schro-dim-3}
\bigg\|-i\int_0^t e^{-i(t-\tau)\Delta}(\rho^2 |v_{\tau}|^2)\, d\tau\bigg\|_{Y(T)}\lesssim \|\rho^2 |v|^2\|_{L^{r'}_T \mathcal{W}^{-s+\eta,l'}} \,
\end{equation}
where
$$(r,l)=(\frac{4}{1+2\theta},\frac{3}{1-\theta})\, $$
is a Schr{\"o}dinger admissible pair. Now, according to Leibniz fractional rule (Lemma~\ref{frac-leibniz}), for all $t\geq 0$,
$$
\|\rho^2 |v|^2(t,.)\|_{\mathcal{W}^{-s+\eta,l'}}\lesssim \| \rho v(t,.)\|_{H^{-s+\eta}}\| \rho v(t,.)\|_{L^{\frac{6}{1+2\theta}}}\, .
$$
Then, by interpolation (see Lemma~\ref{interpol}), the bound below holds true
$$\| \rho v(t,.)\|_{L^{\frac{6}{1+2\theta}}}\leq  \| \rho v(t,.)\|_{H^{-s+\eta}}^\theta \| \rho v(t,.)\|_{\cw^{-s,6}}^{1-\theta} \, $$
and, consequently,
\begin{align*}
\|\rho^2 |v|^2(t,.)\|_{\mathcal{W}^{-s+\eta,l'}}&\lesssim \| \rho v(t,.)\|_{H^{-s+\eta}}^{1+\theta}\| \rho v(t,.)\|_{\cw^{-s,6}}^{1-\theta} \\
&\lesssim \|  v(t,.)\|_{H^{-s+\eta}_\rho}^{1+\theta}\| v(t,.)\|_{\cw^{-s,6}}^{1-\theta}+\| v(t,.)\|_{H^{-s}}^{1+\theta}\| v(t,.)\|_{\cw^{-s,6}}^{1-\theta}   \, ,
\end{align*}
where we have used a commutator estimate (Lemma~\ref{commut}) to obtain the second inequality. H{\"o}lder's inequality with $\la=3-2\theta > 1$ yields
\begin{align*}
&\int_0^T dt \, \|\rho^2 |v|^2(t,.)\|_{\mathcal{W}^{-s+\eta,l'}}^{r'}\\
&\lesssim \bigg( \int_0^T dt \, \|  v(t,.)\|_{H^{-s+\eta}_\rho}^{(1+\theta)r'\la}\bigg)^{\frac{1}{\la}} \bigg(\int_0^T dt \, \| v(t,.)\|_{\cw^{-s,6}}^{(1-\theta)r'\la'}\bigg)^{\frac{1}{\la'}}+\|v\|_{Y(T)}^{r'(1+\theta)} \int_0^T dt \, \|v_t\|^{r'(1-\theta)}_{\cw^{-s,6}} \, .
\end{align*}
According to the assumptions on our parameters (recall that $\eta <\frac14-s$ and $\theta=\frac{s}{\eta}$),
$$(1+\theta)r'\la=4(1+\theta)<\frac{1}{\eta} \quad \text{and} \quad (1-\theta)r'\la' =2 \, ,$$
yielding
\begin{align*}
&\int_0^T dt \, \|\rho^2 |v|^2(t,.)\|_{\mathcal{W}^{-s+\eta,l'}}^{r'}\\
&\lesssim T^{\frac{1-4\eta(1+\theta)}{\la}}\bigg( \int_0^T dt \, \|  v(t,.)\|_{H^{-s+\eta}_\rho}^{\frac{1}{\eta}}\bigg)^{(1+\theta)r' \eta} \|v\|_{Y(T)}^{\frac{2}{\la'}}+T^{1-\frac{r'(1-\theta)}{2}}\|v\|_{Y(T)}^{r'(1+\theta)} \|v\|_{Y(T)}^{r'(1-\theta)} \\
&\lesssim T^{\frac{1-4\eta-4s}{\la}}\|v\|_{Y(T)}^{(1+\theta)r'+\frac{2}{\la'}}+T^{1-\frac{r'(1-\theta)}{2}}\|v\|_{Y(T)}^{2r'}  \, .
\end{align*}
The previous estimate can be reformulated as
$$
\|\rho^2 |v|^2\|_{L^{r'}_T \mathcal{W}^{-s+\eta,l'}}  \lesssim \{T^{\tilde{\varepsilon}}+T^{\frac14}\} \|v\|_{Y(T)}^2 \, ,
$$
with $\tilde{\varepsilon}=\frac{1}{4} (1-4\eta-4s)$.
Coming back to~\eqref{schro-dim-3}, we have finally obtained that
\begin{equation*}
\bigg\|-i\int_0^t e^{-i(t-\tau)\Delta}(\rho^2 |v_{\tau}|^2)\, d\tau\bigg\|_{Y(T)} \lesssim \{T^{\tilde{\varepsilon}}+T^{\frac14}\} \|v\|_{Y(T)}^2 \, .
\end{equation*}

\noindent
\textbf{Bound on $\displaystyle-i\int_0^t e^{-i(t-\tau)\Delta}((\rho \overline{v}_{\tau})\cdot(\luxor_{\tau}))\, d\tau$, $\displaystyle-i\int_0^t e^{-i(t-\tau)\Delta}((\rho v_{\tau})\cdot(\overline{\luxor_{\tau}}))\, d\tau$:}
Repeating the arguments used for $d=2$ (remark indeed that $-s+\eta>s$), we establish that
$$\bigg\|-i\int_0^t e^{-i(t-\tau)\Delta}((\rho \overline{v}_{\tau})\cdot(\luxor_{\tau}))\, d\tau\bigg\|_{Y(T)}+\bigg\|-i\int_0^t e^{-i(t-\tau)\Delta}((\rho v_{\tau})\cdot(\overline{\luxor_{\tau}}))\, d\tau\bigg\|_{Y(T)}\lesssim T^{1-\eta}\|\luxor\|_{L^\infty_T {\mathcal{W}^{-s,r}}} \|v\|_{Y(T)} \, .$$

\noindent
\textbf{Bound on $\cherry_1$:}
Keeping in mind that $\kappa=2-\alpha$ and $\eta < \frac{1}{2}$, $-s+\eta\leq -2s + \kappa$ if $\varepsilon>0$ is small enough and we immediately have that
$$\|\cherry_1\|_{Y(T)}\leq \|\cherry_1\|_{L^{\infty}_T H^{-2s+\kappa}}+\|\cherry_1\|_{L^{\infty}_T \mathcal{W}^{-2s+\kappa,6}}.$$
\noindent
That concludes the proof.

\

\end{document}